\documentclass{amsart}
\usepackage{mathrsfs}
\usepackage{enumitem}
\newtheorem{thma}{Theorem}
\theoremstyle{remark}

\theoremstyle{plain}
\newtheorem{lemma}[thma]{Lemma}

\newcommand{\trace}{\operatorname{trace}}

\newcommand{\field}{\ensuremath{\mathbb{F}}}
\newcommand{\vecspace}[1]{\ensuremath{\mathcal{#1}}}
\newcommand{\mat}[2]{\ensuremath{{#1}^{#2\times #2}}}

\newcommand{\Hom}{\operatorname{Hom}}
\newcommand{\Nilp}{\operatorname{Nilp}}
\newcommand{\closure}[2][3]{%
  {}\mkern#1mu\overline{\mkern-#1mu#2}}
\numberwithin{equation}{section}
\numberwithin{thma}{section}
\usepackage{enumitem}
\makeatletter
\newcommand*{\wackyenum}[1]{%
  \expandafter\@wackyenum\csname c@#1\endcsname%
}

\newcommand*{\@wackyenum}[1]{%
  $\ifcase#1\or1^{o}\or2^{o}\or3^{o}\or 4^{o} \or 5^{o}
    \else\@ctrerr\fi$%
}
\AddEnumerateCounter{\wackyenum}{\@wackyenum}{53.13}
\makeatother

\thanks{The authors were partially supported by NSF grant DMS-0901628. The second author was also
partially supported by US-Israel BSF grant 2010432.}

\title[Implicit/inverse function theorems for free nc functions]{Implicit/inverse
function theorems for free noncommutative functions}
\author[G. Abduvalieva]{Gulnara Abduvalieva}
\address{Department of Mathematics \\
Drexel University\\
3141 Chestnut St.\\
Philadelphia, PA, 19104}
 \email{gka26@drexel.edu}

\author[D.S.~Kaliuzhnyi-Verbovetskyi]{Dmitry S.
Kaliuzhnyi-Verbovetskyi}

\address{Department of Mathematics \\
Drexel University\\
3141 Chestnut St.\\
Philadelphia, PA, 19104}
  \email{dmitryk@math.drexel.edu}
\date{}

\subjclass{47J07; 17A50; 46L07; 16N40}

 \keywords{Free noncommutative functions; implicit
function theorem; inverse function theorem; operator spaces;
nilpotent matrices}

\begin{document}
\maketitle

\begin{abstract} We prove an implicit function theorem and an inverse
function theorem for free  noncommutative functions over operator
spaces and on the set of nilpotent matrices. We apply these
results to study dependence of the solution of the initial value
problem for ODEs in noncommutative spaces on the initial data and
to extremal problems with noncommutative constraints.

\end{abstract}

\section {Introduction and Statements of the Results} \label{sec:intro}

\subsection{Free NC functions}

A free noncommutative (nc) function is a mapping  defined on the set of matrices of all sizes which respects direct sums
 and similarities, or equivalently, respects intertwinings. Examples include but are not limited to nc polynomials, power series, and matrix-valued rational expressions. The theory of free nc  functions was first  introduced  in the
articles of Joseph L. Taylor \cite{JLTaylor,JLTaylor-nc}. It was
further developed by D.-V. Voiculescu \cite{Voic1,Voic2} for the
needs of free probability. Various aspects of nc
functions\footnote{Here and in the rest of the paper we will omit
the word ``free" for short.} have been studied by Helton
\cite{Helton}, Helton, Klep, and McCullough \cite{HKMcC1, HKMcC2},
Helton and McCullough \cite{HMcC},  Helton and Putinar \cite{HP},
Popescu \cite{Pop1, Pop2},  Muhly and Solel \cite{MS}, Agler and
McCarthy \cite{Ag-Mc1, Ag-Mc2}, Agler and Young \cite{Ag-Y},
 the second author  and Vinnikov\cite{ KVV2,KVV1}, and others.

 In the book of the second author and Victor Vinnikov \cite{KV-VV},  the theory has been put on a systematic  foundation. The nc difference-differential calculus has been developed for
studying various questions of nc analysis; in particular,
the classical (commutative) theory of analytic functions was extended
to a nc setting. It has been established that very mild assumptions of local boundedness of nc functions imply analyticity.

We provide the reader with some basic definitions from
\cite{KV-VV}.
 Let $ \mathcal{R} $ be a unital commutative ring. For a module $\mathcal{M}$ over
$\mathcal{R},$ we
 define the \emph{nc space over  $ \mathcal{M} $},
\begin{equation}\label{eq:ncspace}
 \mathcal{M}_{\rm nc}:=\coprod_{n=1}^{\infty}  \mathcal{M}^{n \times n}.
\end{equation}

 A subset $ \Omega  \subseteq  \mathcal{M}_{\rm nc} $ is called a \emph{nc set} if it is closed under direct sums;
that is, denoting $\Omega_{n}=\Omega \cap \mathcal{M}^{n\times n}$, we have
\begin{equation}\label{eq:ncset}
   X \in \Omega_{n},   Y \in \Omega_{m}   \Longrightarrow X \oplus Y := \left [ \begin{array}{cc}
X& 0
\\ 0&Y
  \end{array} \right ] \in\Omega_{n+m}.
\end{equation}

Matrices over $ \mathcal{R}  $  act from the right and from the left on matrices over $\mathcal{M}$
by the standard rules of matrix multiplication: if $ T \in \mathcal{R}^{r \times p} $ and  $ S \in
\mathcal{R}^{p\times s} $, then for $ X \in \mathcal{M}^{p \times p} $ we have
 \[TX \in  \mathcal{M}^{r \times p}, \quad       XS \in  \mathcal{M}^{p \times s}. \]

 In the special case where
$\mathcal{M}=\mathcal{R}^{d}$, we identify matrices over
$\mathcal{M}$  with  $ d $-tuples of matrices over $\mathcal{R}$
:\[ (\mathcal{R}^{d})^{p\times q}\cong( \mathcal{R}^{p\times q})^{d}
.\] Under this identification, for $d$-tuples  $X=(X_{1}, \ldots, X_{d}) \in
(\mathcal{R}^{n\times n})^{d} $ and $Y=(Y_{1}, \ldots, Y_{d}) \in
(\mathcal{R}^{m\times m})^{d},$  their direct sum has the form
\[  X \oplus Y = \left ( \left [ \begin{array}{cc} X_{1}& 0\\ 0 &Y_{1}
 \end{array} \right ], \ldots,
\left [ \begin{array}{cc} X_{d}& 0\\ 0 &Y_{d} \end{array} \right ]
\right)
 \in (\mathcal{R}^{(n+m)\times(n+m)})^{d}; \] and for a $d$-tuple $X=(X_{1}, \ldots, X_{d}) \in
(\mathcal{R}^{p\times p})^{d} $ and matrices $T \in \mathcal{R}^{r\times p}$,
   $S \in \mathcal{R}^{p\times s}$,
\[TX=(TX_{1}, \ldots, TX_{d}) \in (\mathcal{R}^{r\times p})^{d},\qquad XS=(X_{1}S, \ldots, X_{d}S)
 \in (\mathcal{R}^{p\times s})^{d};  \] that is, $T$ and $S$ act on $d$-tuples of matrices componentwise.

 Let $\mathcal {M}$ and $\mathcal{N}$ be modules over
$ \mathcal{R}$, and let $ \Omega \subseteq \mathcal{M}_{\rm nc}$ be a nc set. A mapping \[  f: \Omega \to
\mathcal{N}_{\rm nc}\] with the property that $f(\Omega_{n}) \subseteq \mathcal{N}^{n\times n}$, $n=1,2,\ldots$,
is called  a \emph{nc function} if $f$ satisfies  the following two conditions:
\begin{equation}\label{eq:dirsums}   f \ \emph{respects\ direct\ sums:\ }
f(X\oplus Y) = f(X) \oplus f(Y),\quad X, Y \in \Omega; \end{equation}
\begin{multline}
\label{eq:sim}  f \ \emph{respects\ similarities:}\ {\rm if\ } X
\in \Omega_{n}\ {\rm and}\  S \in \mathcal{R}^{n\times n}\ {\rm
is}\ {\rm invertible\ }  \\{\rm with\ }  SXS^{-1} \in \Omega_{n},\
{\rm then}\
  f(SXS^{-1}) = Sf(X)S^{-1}, \end{multline}
or, equivalently, satisfies the single condition:
\begin{multline}\label{eq:intertw}  f\ respects\ intertwinings:\
{\rm if}\ X \in \Omega_{n}, Y \in \Omega_{m},\ {\rm and}\ T \in \mathcal{R}^{n\times m}\\
{\rm are\ such\ that\ } XT = TY,\ {\rm then}\ f(X)T = Tf(Y).
\end{multline}

We will say that a nc set $\Omega \subseteq \mathcal{M}_{\rm nc}$ is  \emph{right admissible} if for all $X \in \Omega_{n},$  $Y \in \Omega_{m},$ and all $Z \in \mathcal{M}^{n \times m}$ there exists an invertible $r \in \mathcal{R}$ such that
\[
\left [ \begin{array}{cc} X& rZ\\ 0&Y \end{array} \right ] \in \Omega_{n + m}.
\]

Next we define the right nc difference-differential operator $\Delta_{R}.$
Let a nc set $\Omega$ be right admissible, and let $f$ be a nc function on $\Omega,$
  then for every $X \in \Omega_{n},$   $ Y \in \Omega_{m},$ and  $Z \in \mathcal{M}^{n \times m},$
  and for an invertible $r \in \mathcal{R}$ such that
  $\left [ \begin{smallmatrix}X& rZ\\ 0&Y \end{smallmatrix} \right ] \in \Omega_{n + m}$ we define first
  $\Delta_{R}f(X, Y)(rZ)$ as the (1, 2) block of the matrix
  $f\left(\left[\begin{smallmatrix} X& rZ\\ 0&Y \end{smallmatrix} \right] \right),$
 and then
\begin{equation}\label{eq:intro2}
\Delta_{R}f(X, Y)(Z) = r^{-1}\Delta_{R}f(X, Y)(rZ).
\end{equation}
By \cite[Proposition 2.2]{KV-VV}, $\Delta_{R}f(X, Y)(Z)$ is well
defined. By \cite[Propositions 2.4, 2.6]{KV-VV}, $\Delta_{R}f(X,
Y)(\cdot) \colon  \mathcal{M}^{n \times m}  \to \mathcal{N}^{n
\times m}$ is a linear mapping.

Analogously, one can define the left nc difference-differential operator $\Delta_{L}$ via evaluations of nc functions on block lower triangular matrices. For our purposes, it suffices to consider only the ``right" version of the theory.

By \cite[Theorem 2.11]{KV-VV},
if $ f\colon \Omega \to \mathcal{N}_{\rm nc}$ is a nc function on a right admissible nc set $\Omega,$ then for all
$n, m \in \mathbb{N},$ arbitrary  $X \in \Omega_{n}$,  $Y \in \Omega_{m},$ and  $S \in \mathcal{R}^{n \times m}$
we have
\begin{equation}\label{eq:intro5}
Sf(X) - f(Y)S = \Delta_{R}f(X, Y)(SX - YS).
\end{equation}
In particular  \cite[Theorem 2.10]{ KV-VV}), if $m = n, \ S = I_{n},$ then
\begin{equation}\label{eq:intro6}
f(X) - f(Y) = \Delta_{R}f(X, Y)(X - Y).
\end{equation}
Thus  $\Delta_{R}$  plays the role of a nc finite difference operator of the first order.

When $\mathcal{R}$ is a field, $\mathcal{M}^{n \times n}$ and  $\mathcal{N}^{n \times n}$  for $n = 1, 2, \hdots $ are topological vector spaces, and $f$ is continuous, we have that $\Delta_{R}f(Y, Y)$ is the differential of $f$ at $Y.$

By \cite[Propositions 2.15, 2.17, 3.2]{KV-VV}, $\Delta_{R}f(X, Y)(\cdot)$ as a function of $X$ and $Y$ respects direct sums and similarities, or equivalently, respects  intertwinings. That is, if $ f\colon \Omega \to \mathcal{N}_{\rm nc}$ is a nc function on a right admissible nc  set $\Omega  \subseteq \mathcal{M}_{\rm nc}$, then
\begin{multline}\label{eq:intro7}
\Delta_{R}f(X^{\prime} \oplus X^{\prime \prime}, Y^{\prime} \oplus Y^{\prime \prime} )\left (\left [ \begin{array}{cc} Z^{\prime, \prime}& Z^{\prime, \prime \prime}\\
Z^{\prime \prime, \prime}&Z^{\prime \prime, \prime \prime} \end{array} \right ] \right ) \\ =
\left [ \begin{array}{cc} \Delta_{R}f(X^{\prime}, Y^{\prime}  )(Z^{\prime,\prime})&  \Delta_{R}f(X^{\prime}, Y^{\prime \prime}  )(Z^{\prime, \prime \prime})\\
\Delta_{R}f(X^{\prime \prime}, Y^{\prime}  )( Z^{\prime \prime,\prime})& \Delta_{R}f(X^{\prime \prime}, Y^{\prime \prime}  )(Z^{\prime\prime,\prime\prime})\\
 \end{array} \right ]
\end{multline}
for  $n^{\prime},$ $  m^{\prime}\in \mathbb{N}$,  $n^{\prime \prime}$,  $m^{\prime \prime}\in \mathbb{Z}_{+}$,  $X^{\prime} \in \Omega_{n^{\prime}}$,
 $X^{\prime \prime} \in \Omega_{n^{\prime \prime}}$, $Y^{\prime} \in \Omega_{n^{\prime}}$, $Y^{\prime \prime} \in \Omega_{n^{\prime \prime}}$,
 $\left [ \begin{smallmatrix} Z^{\prime, \prime}& Z^{\prime, \prime \prime}\\ Z^{\prime \prime,\prime}& Z^{\prime \prime, \prime \prime} \end{smallmatrix} \right ] \in \mathcal{M}^{{(n^{\prime} + n^{\prime \prime}) \times (m^{\prime} + m^{\prime \prime})}},$ with block entries of appropriate sizes, and if either $n^{\prime \prime}$ or $m^{\prime \prime}$  is $0$, then the corresponding block entries are void;
\begin{equation}\label{eq:intro8}
\Delta_{R} f(TXT^{-1}, SYS^{-1})(TZS^{-1}) = T\Delta_{R} f(X, Y)(Z)S^{-1}
\end{equation}
 for  $n,$ $ m \in \mathbb{N},$   $X \in \Omega_{n},$ $ Y \in \Omega_{m},$  $ Z  \in \mathcal{M}^{n \times m},$ and invertible $T  \in \mathcal{R}^{n \times n},$  $S  \in \mathcal{R}^{m \times m}$ such that \ $TXT^{-1} \in \Omega_{n},$  $ SXS^{-1} \in \Omega_{m};$
or equivalently,
\begin{equation}\label{eq:intro9}
TX = \tilde{X}T, \ YS = S\tilde{Y} \ \Rightarrow   T\Delta_{R} f(X, Y)(Z)S = \Delta_{R}f(\tilde{X}, \tilde{Y})(TZS)
\end{equation}
for  $n,$ $\tilde{n},$ $  m,$  $\tilde{m} \in \mathbb{N},$  $  X \in \Omega_{n},$  $ \tilde{X} \in \Omega_{\tilde{n}},$ $ Y \in \Omega_{m},$  $  \tilde{Y} \in \Omega_{\tilde{m}},  $ $Z \in \mathcal{M}^{n \times m},$ and $ T \in \mathcal{R}^{\tilde{n} \times n},$ $ S \in \mathcal{R}^{m \times \tilde{m}}.$
Thus, for a nc function $f\colon \Omega \to \mathcal{N}_{\rm nc}$ where $\Omega \subseteq \mathcal{M}_{\rm nc},$  $\Delta_{R}f(X, Y)(\cdot)$ is a function of  two arguments $X$ and $Y$ (for all $n,$ $ m \in \mathbb{N}$) with values linear mappings $\mathcal{M}^{n \times m} \to \mathcal{N}^{n \times m}$ that respects direct sums and  similarities.

For $\mathcal{M}_{0}$, $\mathcal{M}_{1}$, $\mathcal{N}_{0}$, $\mathcal{N}_{1}$ modules over a unital
commutative ring $\mathcal{R},$ and $\Omega^{0} \subseteq  \mathcal{M}_{0, \rm nc}$,
$\Omega^{1} \subseteq  \mathcal{M}_{1, \rm nc}$ nc sets, we define a \emph{nc function of order 1} to be a
 function on $\Omega^{0} \times \Omega^{1} $ so that for $X^{0} \in \Omega^{0}_{n_{0}} $ and
 $X^{1} \in \Omega^{1}_{n_{1}}$,
\[f(X^{0}, X^{1}) \colon \mathcal{N}_{1}^{n_{0} \times n_{1}} \to
 \mathcal{N}_{0}^{n_{0} \times n_{1}}\]
  is a linear mapping, and  that
 $f$ respects direct sums and similarities in each argument. This class of nc functions of order 1 is denoted by $\mathcal{T}^{1}(\Omega^{(0)}, \Omega^{(1)}; \mathcal{N}_{0, \rm nc}, \mathcal{N}_{1, \rm nc}).$ We also denote by $\mathcal{T}^{0}(\Omega;{\mathcal{N}}_{\rm nc})$ the class of nc functions (of order 0) $f\colon\Omega\to{\mathcal{N}}_{\rm nc}$. It turns out that for $f \in \mathcal{T}^0(\Omega; \mathcal{N}_{\rm nc}),$ one has $\Delta_{R} f \in \mathcal{T}^{1}(\Omega, \Omega; \mathcal{N}_{ \rm nc}, \mathcal{M}_{ \rm nc}).$

More generally, one can define \emph{ nc functions of order }$k.$   Let $\mathcal{M}_{0}$, \ldots, $\mathcal{M}_{k}$,
 $\mathcal{N}_{0}$, \ldots, $\mathcal{N}_{k}$ be modules over a unital commutative ring $\mathcal{R}.$ Let $\Omega^{0} \subseteq  \mathcal{M}_{0, \rm nc}, \hdots,  \Omega^{k} \subseteq  \mathcal{M}_{k, \rm nc}$ be nc sets.  A nc function of order $k$ is  a function of $k + 1$ arguments on $\Omega^{0} \times \cdots \times \Omega^{k} $ so that for
$X^{0} \in \Omega^{0}_{n_{0}}$, \ldots, $X^{k} \in \Omega^{k}_{n_{k}}$,
\[f(X^{0}, \ldots,
X^{k}) \colon \mathcal{N}_{1}^{n_{0} \times n_{1}}\times \cdots
\times \mathcal{N}_{k}^{n_{k - 1} \times n_{k}} \to  \mathcal{N}_{0}^{n_{0} \times n_{k}}\]
  is a $k$-linear mapping, and  that $f$ respects direct sums and similarities in each argument in a way similar to (\ref{eq:intro7}) -- (\ref{eq:intro9}). This class of nc functions of order $k$ is denoted as $\mathcal{T}^{k}(\Omega^{(0)}, \hdots, \Omega^{(k)}; \mathcal{N}_{0, \rm nc}, \hdots, \mathcal{N}_{k, \rm nc}).$

 One extends $\Delta_{R}$ to an operator from $\mathcal{T}^{k}$ to $\mathcal{T}^{k + 1}$ for all $k.$ Similarly to the case $k = 0,$ it is done by evaluating a nc function of order $k (> 0)$ on a $(k + 1)$-tuple of square matrices with one of the arguments block upper triangular. We have therefore
\begin{multline*}
\Delta_{R} \colon  \mathcal{T}^{k}(\Omega^{(0)}, \hdots \Omega^{(k)}; \mathcal{N}_{0,\rm nc}, \hdots \mathcal{N}_{k,\rm nc})\\ \to  \mathcal{T}^{k + 1}(\Omega^{(0)}, \hdots \Omega^{(k)}, \Omega^{(k)}; \mathcal{N}_{0,\rm nc}, \hdots \mathcal{N}_{k,\rm nc}, \mathcal{M}_{k,\rm nc}),
\end{multline*}
for $k = 0, 1, \hdots.$ Iterating this operator $\ell$ times, we obtain the $\ell$-\emph{th order nc difference-differential operator}

\begin{multline*}
\Delta_{R}^{\ell} \colon  \mathcal{T}^{k}(\Omega^{(0)}, \hdots \Omega^{(k)}; \mathcal{N}_{0,\rm nc}, \hdots \mathcal{N}_{k,\rm nc})\\ \to  \mathcal{T}^{k + 1}(\Omega^{(0)}, \hdots \Omega^{(k)}, \underbrace{\Omega^{(k)}, \hdots , \Omega^{(k)}}_{\ell \ {\rm times}}; \mathcal{N}_{0,\rm nc}, \hdots \mathcal{N}_{k,\rm nc}, \underbrace{\mathcal{M}_{k,\rm nc}, \hdots, \mathcal{M}_{k,\rm nc}}_{\ell \ {\rm times}}),
\end{multline*}
for $k = 0, 1, \hdots.$ According to the definition, $\Delta^{\ell}_{R}f$ is calculated iteratively by evaluating nc functions of increasing orders on $2 \times 2$ block upper triangular matrices at each step. It turns out \cite[Theorems 3.11, 3.12]{KV-VV} that $\Delta^{\ell}_{R}f$ can also be calculated in a single step by evaluating $f$ on $(\ell + 1) \times (\ell + 1) $ block upper bidiagonal matrices.  If  $ f: \Omega \to \mathcal{N}_{\rm nc}$ is a nc function (of order $0$), where $\Omega \subseteq \mathcal{M}_{nc}$  is  a right admissible nc set, and  if $X^{0}, \hdots, X^{\ell} \in \Omega, $ then
\begin{multline*}
 \Delta^{\ell}_{R}f(X^{0}, \hdots, X^{\ell})(Z^{1}, \hdots, Z^{\ell})
\mathrel {\mathop:} = f\left (\left [ \begin{array}{ccccc} X^{0}& Z^{1} & 0 & \hdots & 0\\ 0&X^{1} & \ddots & \ddots & \vdots\\ \vdots&\ddots & \ddots & \ddots & Z^{\ell} \\ 0&\hdots & \hdots & 0 & X^{\ell}\end{array} \right ] \right )_{1, \ell + 1}
\end{multline*}
extends to a $\ell $-linear mapping in $Z^{1}, \hdots, Z^{\ell}.$

The following theorem \cite[Theorem 4.1]{KV-VV}) derives an nc analogue of the  Taylor formula, which is called the  Taylor--Taylor (TT) formula after Brook Taylor and Joseph L. Taylor.
\begin{thma}[The Taylor--Taylor Formula]\label{eq:TT}
Let $f \in \mathcal{T}^{0}(\Omega; \mathcal{N}_{\rm nc})$ with $\Omega \subseteq \mathcal{M}_{\rm nc}$ a right admissible nc set, $n \in \mathbb{N},$ and $Y \in \Omega_{s}.$ Then for each $N \in \mathbb{N},$  and $X \in \Omega_{s},$
\begin{multline}\label{eq:intro10}
f(X) = \sum_{\ell = 0}^{N}\Delta_{R}^{\ell}f \underbrace{(Y, \dots, Y)}_{\ell + 1 \ {\rm times}}\underbrace{(X - Y ), \hdots , (X - Y)}_{\ell \ {\rm times}}\\
+ \Delta_{R}^{N + 1}f \underbrace{(Y, \dots, Y, X)}_{N + 1 \ {\rm times}}\underbrace{(X - Y), \hdots, (X - Y)}_{N + 1 {\rm \ times}}.
\end{multline}

\end{thma}

\subsection{The setting of operator spaces}

Next we give the definition of an operator space; see \cite{ER}, \cite{G.Pi}, \cite{Paulsen}, and \cite{Ruan}.   Let $\field$ be a field, $\field=\mathbb{C}$ or
$\field=\mathbb{R}$.
 A vector space $\vecspace{W}$ over $\field$ is called an \emph{operator space} if a
 sequence of Banach--space norms $\|\cdot\|_n$ on $\mat{\vecspace{W}}{n}$, $n=1,2,\ldots,$ is
defined so that the following two
 conditions hold:
 \begin{enumerate}
     \item [1.]For every $n,m\in\mathbb{N}$, $X\in\mat{\vecspace{W}}{n},$ and
     $Y\in\mat{\vecspace{W}}{m}$,
     \begin{equation}\label{eq:opscond1}
 \| X\oplus Y\|_{n+m}=\max\{\|X\|_n,\| Y\|_m\};
     \end{equation}
     \item[2.] For every $n\in\mathbb{N}$, $X\in\mat{\vecspace{W}}{n},$ and
     $S,T\in\mat{\field}{n}$,
     \begin{equation*}
 \| SXT\|_{n}\le\|S\|\,\|X\|_n\|T\|,
     \end{equation*}
    where $\|\cdot\|$ denotes the $(2,2)$ operator norm on
    $\mat{\field}{n}$.
 \end{enumerate}

Let  $\vecspace{W}$ be an operator space. For $Y \in \vecspace{W}^{s \times s}$ and $r > 0,$ define a
\emph{nc ball centered at $Y$ of radius $r$} as
\begin{multline*}\label{eq:intro11}
B_{\rm nc}(Y, r) := \coprod_{m=1}^{\infty} B\Big (Y^{(m)}, r\Big)
= \coprod_{m=1}^{\infty}\left \{ X \in \vecspace{W}^{sm \times sm}
\colon \left \| X - Y^{(m)}\right \|_{sm}  < r\right \},
\end{multline*}
where $Y^{(m)}\mathrel {\mathop:} =  \bigoplus_{ 1}^{m}Y \cong I_{ m} \otimes Y.$ Clearly, nc balls are nc sets.
By \cite[Proposition 7.12]{KV-VV},  the nc balls form a basis for a topology on $\vecspace{W}_{\rm nc}.$ This topology is called the uniformly-open topology. Open sets in the uniformly-open topology on $\vecspace{W}_{\rm nc}$ are called \emph{uniformly open}. Notice that uniformly open nc sets are right admissible.

Let $\vecspace{V},  \vecspace{W}$ be  operator spaces, and let $\Omega \subseteq \vecspace{V}_{\rm nc}$ be a uniformly open nc set. A nc function $f\colon \Omega \to \vecspace{W}_{\rm nc} $ is called \emph{uniformly locally bounded} if for any $s \in \mathbb{N}$ and $Y \in \vecspace{V}^{s \times s}$ there exists  $r > 0$ such that $B_{\rm nc}(Y, r) \subseteq \Omega$ and $f$ is bounded on $B_{\rm nc}(Y, r),$ i.e., there is  $M > 0 $ such that $\left \| f(X)\right \|_{sm}  < M$ for all $m \in \mathbb{N}$ and $X \in B_{\rm nc}(Y, r)_{sm}.$  A nc function $f\colon \Omega \to \vecspace{W}_{\rm nc} $ is called \emph{uniformly analytic} if $f$  is uniformly locally bounded and Gateaux (G-)differentiable. A nc function   $f \colon \Omega \to \mathcal{W}_{nc} $ is called G-differentiable if for every $n \in \mathbb{N}$  the function $f|_{\Omega_{n}}$ is G-differentiable, i.e., for every  $X \in \Omega_{n}$ and $ Z \in \mathcal{V}^{n \times n}$ the G-derivative of $f$ at $X$ in the direction $Z,$
\begin{equation*}
\delta f(X)( Z) = \lim_{t \to 0} \frac{f(X+ tZ) - f(X)}{t} = \frac{d}{dt}f(X + tZ)\Big|_{t = 0},
\end{equation*}
exists. We note that $\delta f(X)( Z) = \Delta_{R}f(X, X)(Z),$ see \cite[Theorem 7.2]{KV-VV}.
Theorem \ref{thm:intro2} below \cite[Corollary 7.26]{KV-VV}) states that a uniformly locally bounded nc function is uniformly analytic, and its TT series converges uniformly and absolutely on uniformly open nc balls.

We will use the following notations. By $\odot_{s}$ we denote the multiplication of matrices over the tensor algebra
\[ {\bf T} (\mathcal{M}^{s \times s}) =  \bigoplus_{ \ell = 0}^{\infty} (\mathcal{M}^{s \times s})^{\otimes\ell} \]
where $\mathcal{M}$ is a module over a unital commutative ring $\mathcal{R}.$  That is, if
\[ X \in \left [(\mathcal{M}^{s \times s})^{\otimes\ell}\right]^{n \times m} \cong \mathcal{M}^{s^{\ell}n \times s^{\ell}m } \,\, \,{\rm and} \,\,\,
Y \in \left [(\mathcal{M}^{s \times s})^{\otimes r}\right]^{m \times p} \cong \mathcal{M}^{s^{r}m \times s^{r}p },\]
then we have
\[\Big(X \odot_{s}Y\Big)_{ij} = \sum_{k = 1}^{m}X_{ik} \otimes Y_{kj} \in (\mathcal{M}^{s \times s})^{\otimes (\ell + r)} \]
so that
\[X \odot_{s}Y \in \left [(\mathcal{M}^{s \times s})^{\otimes (\ell + r)}\right ] ^{n \times p} \cong  \mathcal{M}^{s^{\ell + r}n \times s^{\ell + r}p }.\]

We will identify multilinear forms $g \colon (\mathcal{M}^{s \times s})^{\ell} \to \mathcal{N}^{s \times s}$ with linear mappings $g \colon (\mathcal{M}^{s \times s})^{\otimes\ell} \to \mathcal{N}^{s \times s},$ and write $(Z^{1} \odot_{s} \cdots  \odot_{s} Z^{\ell})g \mathrel {\mathop:} =g(Z^{1}, \hdots, Z^{\ell}).$ We  then extend $g$ to matrices $A$ over ${\bf T} (\mathcal{M}^{s \times s})$ by defining $g(A)_{ij} = g(A_{ij})$ so that
\[g \colon \left[(\mathcal{M}^{s \times s})^{\otimes\ell} \right]^{m \times m}\to (\mathcal{N}^{s \times s})^{m \times m} \cong \mathcal{N}^{sm \times sm}.\]
In particular, the terms in the TT series centered at $Y \in \mathcal{M}^{s \times s}$ are written as
\[\Big (X - Y^{(m)}\Big )^{\odot_{s} \ell}\Delta_{R}^{\ell}f \underbrace{(Y, \dots, Y)}_{\ell + 1 \ {\rm times}} = \Delta_{R}^{\ell}f \underbrace{(Y, \dots, Y)}_{\ell + 1 \ {\rm times}}\underbrace{\Big (X - Y^{(m)}, \ldots, X - Y^{(m)}\Big )}_{\ell \ {\rm times}},\]
where $X \in \mathcal{M}^{sm \times sm};$
see \cite[Chapter 4]{KV-VV} for details.

\begin{thma}\cite[Corollary 7.26]{KV-VV} \label{thm:intro2}
Let a nc function  $f\colon \Omega \to \vecspace{W}_{\rm nc} $ be uniformly locally bounded. Let  $s \in \mathbb{N},$  $ Y \in \Omega_{s}, $ and let
\[\delta  \mathrel {\mathop:} = \sup \{r > 0\colon f \,\, {\rm is \,\, bounded \,\, on \,\,  B_{\rm nc}(Y, r)}\}.\]  Then
\[
 f(X) = \sum_{\ell = 0}^{\infty}\Big (X - Y^{(m)}\Big )^{\odot_{s} \ell}\Delta_{R}^{\ell}f \underbrace{(Y, \dots, Y)}_{\ell + 1 \ {\rm times}}
\]
holds, with the TT series convergent absolutely and uniformly  on every  nc ball $B_{\rm nc}(Y, r)$ with $r < \delta.$ Moreover,
\[
\sum_{\ell = 0}^{\infty} \sup_{m \in \mathbb{N}, X \in B_{\rm nc}(Y, r)_{sm}}\Big \| \Big (X - Y^{(m)}\Big )^{\odot_{s} \ell}\Delta_{R}^{\ell}f \underbrace{(Y, \dots, Y)}_{\ell + 1 \ {\rm times}} \Big\|_{sm} < \infty.
\]
\end{thma}

Let $\mathcal{V}$ and $\mathcal{W}$ be operator spaces. For an operator $A\colon \mathcal{V} \to \mathcal{W}$ we set
$A^{(n)} ={\rm id}_{n} \otimes A,$  where ${\rm id}_{ n}: \mathbb{F}^{n \times n} \rightarrow  \mathbb{F}^{n \times n}$ is the identity operator,  i.e.,  ${\rm id}_{ n}X = X,$   so that  $A^{(n)}$ can be identified with an operator from $\mathcal{V}^{n \times n}$ to $ \mathcal{W}^{n \times n}$ as follows
\[\left[ A^{(n)}(X)\right ]_{ij} = A(X_{ij}).\]

We say that $A$  is completely bounded if $\left\|  A^{(n)} \right\| \leq C$ for all $n \in \mathbb{N},$ and a  constant $C$ is  independent of $n.$

The space of completely bounded operators $A$ from $\mathcal{V}$
to $\mathcal{W}$ is denoted by $\mathcal{L}_{\rm cb}(\mathcal{V},
\mathcal{W}),$ where the norm of $A$ is given by $\left\|  A
\right\|_{\mathcal{L}_{\rm cb}(\mathcal{V}, \mathcal{W})} =
\sup_{n \in \mathbb{N}}\left\|  A^{(n)} \right\|.$

Our main results in the setting of operator spaces are the following two theorems.
\begin{thma}[Implicit nc function theorem] \label{thm:ops-1}
Let $\mathcal{X}, \mathcal{Y},$ and $\mathcal{Z}$ be operator spaces over the field $\mathbb{F},$ $\mathbb{F} = \mathbb{C}$ or $\mathbb{F} = \mathbb{R},$
 and $\Omega$   a uniformly open nc set in $(\mathcal{X}\times \mathcal{Y})_{\rm nc}.$
 Let $s \in \mathbb{N}$ and
  $(X^{0}, Y^{0})\in \Omega_s$. Let $F\colon \Omega  \rightarrow {\mathcal{Z}}_{\rm nc}$ be a nc function satisfying the following conditions:
\begin{enumerate}
\item [1.]\label{eq:cond1}
$F(X^{0}, Y^{0}) = 0.$
 \item[2.] \label{eq:cond2}
$F$ is continuous  at  $(X^{0}, Y^{0})$ with respect to the uniformly--open topologies on $(\mathcal{X}\times \mathcal{Y})_{\rm nc}$ and $\mathcal{Z}_{\rm nc}.$
\item [3.]\label{eq:cond3}
\[\delta^{Y}F(X^{0}, Y^{0}) \in \mathcal{L}_{\rm cb}(\mathcal{Y}^{s \times s}, \mathcal{Z}^{s \times s})\] is  invertible and \[\delta^{Y}F(X^{0}, Y^{0})^{-1} \in \mathcal{L}_{\rm cb}(\mathcal{Z}^{s \times s}, \mathcal{Y}^{s \times s}),\]
where
\begin{equation*}
\delta^{Y} F(X^{0}, Y^{0})( Z) = \lim_{t \to 0} \frac{F( X^{0}, Y^{0}+ tZ) - F(X^{0}, Y^{0})}{t}.
\end{equation*}
\end{enumerate}
Then:
\begin{enumerate}
\item [I.]\label{eq:concl1}
There exist $\alpha,$ $\beta  > 0$ such that for every $ m \in \mathbb{N},$
\[ B(X^{0(m)}, \alpha) \times B(Y^{0(m)}, \beta) \subset \Omega_{sm}.\]

\item [II.]\label{eq:concl2}
There exists a nc function $ f  \colon B_{\rm nc}(X^{0}, \alpha) \rightarrow B_{\rm nc}(Y^{0},\beta), $ such that
\begin{enumerate}
\item For $(X, Y) \in  B(X^{0(m)}, \alpha) \times B(Y^{0(m)}, \beta),$
\[F(X,Y) = 0 \,\, {\rm if \,\, and \,\, only \,\,  if } \,\,  Y = f (X).\]
\item $f$ is uniformly analytic on $B_{\rm nc}(X^{0}, \alpha).$
\end{enumerate}
\item [III.]\label{eq:concl3}
For every $X \in B_{\rm nc}(X^{0}, \alpha), $ the operator $\delta^{X} F(X, f(X))$ has a completely bounded inverse, and
\[\delta f(X) = - \Big(\delta^{Y}F(X, f(X))\Big)^{-1}\delta^{X}F(X, f(X)),\]
where
\begin{equation*}
\delta^{X} F(X, Y)(Z) = \lim_{t \to 0} \frac{F( X+ tZ, Y) - F(X, Y)}{t}.
\end{equation*}
\end{enumerate}
\end{thma}

\begin{thma}[Inverse nc function theorem]\label{thm:ops-2}
 Let $\mathcal{X}$ and $\mathcal{Y}$ be operator spaces,
and  $\Omega^{Y}$  a uniformly open  nc set in $\mathcal{Y}_{\rm nc}.$
 Let $s \in \mathbb{N}$
 and $Y^{0}\in \Omega^{Y}_{s}.$
 Let $g\colon \Omega^{Y}  \rightarrow {\mathcal{X}}_{\rm nc}$ be a nc function satisfying the following conditions:
\begin{enumerate}
\item [1.]$g$ is uniformly analytic on $\Omega^{Y},$ \item [2.] $
\delta g(Y^{0})$  is invertible and $(\delta g(Y^{0}))^{-1} \in
\mathcal{L}_{\rm cb}(\mathcal{X}^{s \times s}, \mathcal{Y}^{s
\times s}). $
\end{enumerate}
Then there exist a uniformly open nc neighborhood $\Delta$ of
$Y^{0}$ in $\mathcal{Y}_{\rm nc}$, and a uniformly open nc
neighborhood $\Gamma$ of $X^{0}=g(Y^0)$ in $\mathcal{X}_{\rm nc}$
such that
\begin{enumerate}
\item[I.]  \label{eq:9} The mapping $g|_{\Delta} \colon \Delta \rightarrow \Gamma$ is a homeomorphism.

\item [II.]\label{eq:10}The mapping $f \colon \Gamma \rightarrow \Delta$, the inverse of $g$, is a uniformly analytic nc function, and
$\delta f(X) = \Big(\delta g(f(X))\Big)^{-1}$
for every $X \in \Gamma$.
\end{enumerate}
\end{thma}

We note that in the case where each of the matrices $X^0$ and
$Y^0$ are multiple copies of a single matrix, i.e.,
$X^0=X_*^{(N)}$ and $Y^0=Y_*^{(N)}$ for some $N\in\mathbb{N}$,
using the methods from \cite{AK} one can extend the underlying nc
functions $F$ and $g$ in Theorems \ref{thm:ops-1} and
\ref{thm:ops-2} to nc functions $\widetilde{F}$ and
$\widetilde{g}$ defined on some neighborhood of $(X_*,Y_*)$
(resp., of $Y_*$), even in the case where the domain
$\Omega_{s/N}$ (resp., $\Omega^Y_{s/N}$) is empty, together with
the assumptions of the theorems. Then the conclusions of Theorems
\ref{thm:ops-1} and \ref{thm:ops-2} can be extended accordingly.
We leave the details to the reader.

\subsection{The setting of nilpotent matrices}

Let $\mathcal{M}$ be a module over a unital commutative ring $\mathcal{R}.$
For $n,$ $ \kappa \in \mathbb{N},$ we denote by  $\Nilp(\vecspace{M}; n, \kappa)$  the  set of $n \times n$ matrices $X$ over $\vecspace{M}$ which are \emph{nilpotent of rank at most $\kappa$,} i.e., $X^{\odot \ell} = 0$ for all $\ell \geq \kappa.$ Here $\odot = \odot_{1}.$ The set of all nilpotent $n \times n $ matrices over $\vecspace{M}$ is denoted by $\Nilp(\vecspace{M}; n) = \bigcup^{\infty}_{\kappa = 1} \Nilp(\vecspace{M}; n, \kappa), $ and the set of all nilpotent matrices over $\vecspace{M}$ is denoted by $\Nilp(\vecspace{M}) = \coprod^{\infty}_{n = 1} \Nilp(\vecspace{M}; n).$

For $Y \in \mathcal{M}^{s \times s},$
we denote by $\Nilp(\vecspace{M}, Y; sm, \kappa)$ the set of matrices $X \in \mathcal{M}^{sm \times sm}$ that are \emph{nilpotent about $Y$   of rank at most $\kappa$}, i.e., $\Big (X - Y^{(m)}\Big )^{\odot_{s} \ell} = 0$ for all $\ell \geq \kappa,$  where $s,$ $m,$  and $\kappa \in \mathbb{N}.$  In other words, $X \in \Nilp(\vecspace{M}, Y; sm, \kappa)$ means that  $X - Y^{(m)} \in \Nilp(\vecspace{M}^{s \times s}; m, \kappa).$  Clearly, $\Nilp(\vecspace{M}; n, \kappa) = \Nilp(\vecspace{M}, 0; n, \kappa).$ We will also use the  corresponding notations:
\[\Nilp(\vecspace{M}, Y; sm) =  \bigcup^{\infty}_{\kappa = 1} \Nilp(\vecspace{M}, Y; sm, \kappa)\]
  for the set of   $sm \times sm $ matrices nilpotent about $Y,$   and
\[\Nilp(\vecspace{M}, Y)= \coprod^{\infty}_{m = 1}\Nilp (\vecspace{M}, Y; sm)\]
  for the set of  nilpotent matrices  about $Y.$

Notice that $\Nilp(\vecspace{M}, Y)  \subseteq \vecspace{M}_{\rm nc}$ is a right admissible nc set, and  $\Nilp(\vecspace{M}, Y)_{sm} = \Nilp(\vecspace{M}, Y; sm).$

A nc function on  $\Nilp(\vecspace{M}, Y)$  is a sum of its TT series.
\begin{thma}\cite[Theorem 5.6]{KV-VV}
Let $Y \in \vecspace{M}^{s \times s} $ and  let $f \colon \Nilp(\vecspace{M}, Y) \to \vecspace{N}_{\rm nc}$ be a nc function. Then for all $X \in \Nilp(\vecspace{M}, Y; sm) $
\begin{equation}\label{eq:intro12}
f(X) =   \sum_{\ell = 0}^{\infty}\Big (X - Y^{(m)}\Big )^{\odot_{s} \ell}\Delta_{R}^{\ell}f \underbrace{(Y, \dots, Y)}_{\ell + 1 \ {\rm times}},
\end{equation}
where the sum has finitely many nonzero terms.
\end{thma}
It follows  from  \cite[Theorem 5.6]{KV-VV} that $f(X) \in \Nilp(\vecspace{N}, f(Y))$ for $X \in \Nilp(\vecspace{M}, Y).$

Our main results in the setting of nilpotent matrices are the following two theorems.

\begin{thma} [Implicit nc function theorem] \label{thm:nilp1}
Let $\mathcal{M}, $ $\mathcal{N},$ and $\mathcal{O}$ be modules over a commutative ring $\mathcal{R},$
and $X^{0} \in \mathcal{M}^{s \times s}$, $ Y^{0} \in \mathcal{N}^{s \times s}$ for $s \in \mathbb{N}.$
Let $F\colon \Nilp(\mathcal{M} \times \mathcal{N}, (X^{0}, Y^{0}))  \rightarrow {\mathcal{O}}_{\rm nc}$ be a nc function satisfying the following conditions:
\begin{enumerate}
\item [1.]$F(X^{0}, Y^{0}) = 0$.
\item [2.]
$
\Delta^{Y}_{R}F((X^{0}, Y^{0}), (X^{0}, Y^{0})) \in \Hom(\mathcal{N}^{s \times s}, \mathcal{O}^{s \times s})$  is invertible. Here
\[\Delta^{Y}_{R}F((X^{0}, Y^{0}), (X^{0}, Y^{0}))(Z) \mathrel {\mathop:}= \Delta_{R}F((X^{0}, Y^{0}), (X^{0}, Y^{0}))(0, Z).\]
\end{enumerate}
Then:
\begin{enumerate}
\item [I.] There exists  a nc function $f\colon \Nilp(\mathcal{M}, X^{0}) \rightarrow \Nilp(\mathcal{N}, Y^{0}) \subset \mathcal{N}_{\rm nc}$
such that
\begin{enumerate}
\item $ X \in \Nilp(\mathcal{M}, X^{0})  \iff (X, f(X)) \in \Nilp(\mathcal{M} \times \mathcal{N}, (X^{0}, Y^{0})) .$
\item $F(X, Y) = 0 \iff  Y = f(X)$.
\end{enumerate}
\item[II.] For any $ X \in \Nilp(\mathcal{M}, X^{0}),$ the linear mapping  $\Delta_{R}^{Y}F((X, f(X)), (X,f(X)))$ is invertible, and
\begin{multline*}
\Delta_{R}f(X, X)(Z)= -\Big(\Delta_{R}^{Y}F((X, f(X)), (X,f(X)))\Big)^{-1}\\
\cdot\Delta_{R}^{X}F\Big((X, f(X)), (X,f(X))\Big)(Z),\end{multline*}
where
$\Delta^{X}_{R}F((X, Y), (X, Y))(Z) \mathrel {\mathop:}= \Delta_{R}F((X, Y), (X, Y))(Z, 0).$
\end{enumerate}
\end{thma}

\begin{thma} [Inverse nc function theorem]\label{thm:nilp3}
Let $\mathcal{M},$ $ \mathcal{N}$  be modules over a unital commutative ring $\mathcal{R},$
and $Y^{0} \in \mathcal{N}^{s \times s}.$
 Let $g\colon \Nilp(\mathcal{N}, Y^{0})  \rightarrow {\mathcal{M}}_{\rm nc}$ be a nc function such that
$
\Delta_{R}g(Y^{0}, Y^{0})$  is invertible.
Then:
\begin{enumerate}
\item [I.] The mapping  $ g\colon \Nilp(\mathcal{N}, Y^{0}) \rightarrow \mathcal{M}_{\rm nc}$ is one-to-one and $g( \Nilp(\mathcal{N}, Y^{0})) = \Nilp(\mathcal{M}, X^{0}).$
\item[II.] The inverse of $g,$ $ f\mathrel {\mathop:}= g^{-1},$ is a nc function mapping $\Nilp(\mathcal{M}, X^{0}))$  onto $\Nilp(\mathcal{N}, Y^{0}).$
\item[III.] $\Delta_{R}g(f(X), f(X))$ is invertible for every $X\in \Nilp(\mathcal{M}, X^{0}),$ and
\[
\Delta_{R}f(X, X) = \Big(\Delta_{R}g(f(X), f(X))\Big)^{-1}.
\]
\end{enumerate}
\end{thma}

\subsection{Remarks}

 Similar theorems have been proved by Pascoe \cite{Pascoe2} and by Agler and McCarthy \cite{Ag-Mc2}.
In \cite{Pascoe2}, an inverse function theorem has been
established for a nc function $g$ on a general nc set $\Omega$
closed under similarities under the assumption that $\Delta_Rg$ is
invertible \emph{everywhere} on $\Omega$. However, the conclusions
are also global. The paper also includes an interesting free
version of partial results on the Jacobi conjecture. In
\cite{Ag-Mc2}, implicit/inverse function theorems similar to our
Theorems \ref{thm:ops-1} and \ref{thm:ops-2} have been proved in
the setting of the \emph{finite-dimensional operator space}
$\mathbb{C}^d$, and the proof uses an essentially
finite-dimensional argument. However, the payoff is some
additional insights in more concrete situations, e.g., an elegant
result that in some generic sense two matrices satisfying a
polynomial equation must commute. Finally, other versions of
implicit/inverse function theorems in different nc settings appear
in \cite{JLTaylor-nc} and \cite{Voic2}.

\section {Proofs of the main results}\label{sec:proof1}

\begin{proof}[Proof of Theorem \ref{thm:ops-1}]
Let
\begin{equation}
\label{eq:family}
g_X(Y) \mathrel {\mathop:} = Y - \Big(\delta^{Y}F(X^{0}, Y^{0}\Big)^{-1 (m)} F(X, Y),  \quad m \in \mathbb{N}, \, (X, Y) \in \Omega_{sm}.
\end{equation}

By  \cite[Corollary 7.28]{KV-VV} there exists $\delta > 0$ such that $F$ is uniformly analytic on
$B_{\rm nc}((X^{0}, Y^{0}), \delta)$. It is obvious that for a fixed $m,$ and  $X \in B_{\rm nc}(X^{0}, \delta)_{ sm}$, the function  $g_{X} $ is defined for every $Y \in B_{\rm nc}(Y^{0}, \delta)_{ sm}$. Here
\[ B_{\rm nc}((X^{0},Y^{0}), \delta)_{sm} = B_{\rm nc}(X^{0}, \delta)_{sm} \times  B(Y^{0}, \delta)_{ sm}.\]

Clearly, $Y_{X}$ is a fixed point of $g_{X}$ if and only if $F(X, Y_{X}) = 0$.

We will show that  there exists a positive number $ \gamma < \delta$ such that for any $X \in B_{\rm nc}(X^{0}, \gamma)$ the mapping $g_{X}$  of the ball $B_{\rm nc}(Y^{0}, \gamma)_{sm_{X}}$ into $\mathcal{Y}^{sm_{X} \times sm_{X}}$ is contractive with the coefficient not exceeding   $\frac{1}{2}$ (here $m_{X} $ is the size of $X).$
 By \cite[Theorem 7.51,  Theorem 7.53]{KV-VV}, and the paragraph following the proof of \cite[Theorem 7.53]{KV-VV}, we have that $\delta^{Y}F$ is continuous on  $ B_{\rm nc}((X^{0},Y^{0}),  \delta)$ and that $\delta^{Y}F(X,Y) \in \mathcal{L}_{\rm cb}\left (\mathcal{Y}^{sm_{(X, Y)} \times sm_{(X, Y)}}, \mathcal{Z}^{sm_{(X, Y)} \times sm_{(X, Y)}}\right )$ for every $(X, Y)$ $\in B_{\rm nc}((X^{0},Y^{0}),  \delta)$.
Next, let $m \in \mathbb{N},$ $X \in B_{\rm nc}(X^{0}, \delta).$ Then for every   $Y \in B_{\rm nc}(Y^{0}, \delta)$  we have:
\[\delta g_X(Y) = I_{Y} - \Big(\delta^{Y}F(X^{0}, Y^{0})\Big)^{-1 (m)} \delta^{Y} F(X, Y).\]
We rewrite the last expression as
\[\delta g_X(Y) = \Big(\delta^{Y}F(X^{0}, Y^{0})\Big)^{-1 (m)}\Big(\delta^{Y}F(X^{0}, Y^{0})^{ (m)} -  \delta^{Y} F(X, Y)\Big).\]
We have

\begin{equation*}
\left\| \Big(\delta^{Y}F(X^{0}, Y^{0})\Big)^{-1}\right\|_ {
\mathcal{L}_{\rm cb}(\mathcal{Z}^{s \times s}, \mathcal{Y}^{s
\times s})}  =
 \sup_{k \in \mathbb{N}}\left\| \Big(\delta^{Y}F(X^{0}, Y^{0})\Big)^{-1 (k)}\right\|_ { \mathcal{L}(\mathcal{Z}^{sk \times sk}, \mathcal{Y}^{sk \times sk})}.
\end{equation*}
Hence condition $\Big(\delta^{Y}F(X^{0}, Y^{0})\Big)^{-1} \in
\mathcal{L}_{\rm cb}(\mathcal{Z}^{s \times s}, \mathcal{Y}^{s
\times s})$  implies
\begin{equation}\label{eq:est1}
\left\| \Big(\delta^{Y}F(X^{0}, Y^{0})\Big)^{-1 (m)}\right\|_ {
\mathcal{L}(\mathcal{Z}^{sm \times sm}, \mathcal{Y}^{sm \times
sm})}  \leq M,
   \end{equation}
where $M \mathrel {\mathop:} = \left\| \Big(\delta^{Y}F(X^{0},
Y^{0})\Big)^{-1 }\right\|_ { \mathcal{L}_{\rm cb}(\mathcal{Z}^{s
\times s}, \mathcal{Y}^{s \times s})} < \infty$. Second, we
estimate
\begin{multline*}
\left\|\delta g_{X}(Y)\right\|_ { \mathcal{L}(\mathcal{Y}^{sm
\times sm})} \leq
 \left\| \Big(\delta^{Y}F(X^{0}, Y^{0})\Big)^{-1 (m)}\right\|_ { \mathcal{L}(\mathcal{Z}^{sm \times sm}, \mathcal{Y}^{sm \times sm})} \\  \times \left\|\delta^{Y}F(X^{0}, Y^{0})^{(m)} - \delta^{Y}F(X, Y) \right\|_ { \mathcal{L}(\mathcal{Y}^{sm \times sm}, \mathcal{Z}^{sm \times sm})} \\
\leq M\left\|\delta^{Y}F(X^{0}, Y^{0})^{(m)} - \delta^{Y}F(X, Y)
\right\|_ { \mathcal{L}(\mathcal{Y}^{sm \times sm},
\mathcal{Z}^{sm \times sm})}.
\end{multline*}
We notice that
\[\delta^{Y}F(X^{0}, Y^{0})^{ (m)} = \delta^{Y}F \Big ( X^{0(m)}, Y^{0(m)} \Big),\]
and \[\Big(\delta^{Y}F(X^{0}, Y^{0})\Big)^{-1 (m)} = \Big (\delta^{Y}F(X^{0(m)}, Y^{0(m)}) \Big)^{-1}.\]
Since $\delta^{Y}F$ is continuous at $(X^{0}, Y^{0})$, there exists $\gamma$, $0 < \gamma < \delta,$ such that for any $(X, Y) \in B_{\rm nc}((X^{0}, Y^{0}), \gamma),$ we have
\begin{multline}\label{eq:est2}
\left\|\delta^{Y}F(X^{0}, Y^{0})^{(m_{(X, Y)})} - \delta^{Y}F(X,
Y) \right\|_ { \mathcal{L}(\mathcal{Y}^{sm_{(X, Y)} \times sm_{(X,
Y)}}, \mathcal{Z}^{sm_{(X, Y)} \times sm_{(X, Y)}})} \\ \leq
\frac{1}{2M}.
\end{multline}
Now using (\ref{eq:est1}) and (\ref{eq:est2}), we  obtain
\begin{equation}\label{eq:est3}
\left\|\delta g_{X}(Y)\right\|_ { \mathcal{L}(\mathcal{Y}^{sm_{(X,
Y)} \times sm_{(X, Y)}})} \leq M \frac{1}{2M} = \frac{1}{2}
\end{equation}
for any $(X, Y) \in B_{\rm nc}((X^{0}, Y^{0}), \gamma). $   From now on, we assume that
 \[(X, Y) \in B_{\rm nc}((X^{0}, Y^{0}), \gamma), \]
 so that  (\ref{eq:est3}) holds.

 Next, using (\ref{eq:est3}) and the mean value theorem for functions in Banach spaces \cite{Kant-Ak}, for any  $ m \in \mathbb{N}, X \in B_{\rm nc}(X^{0}, \gamma)_{sm}$ and any $Y_{1}, Y_{2} \in B_{\rm nc}(Y^{0}, \gamma)_{sm}$, we have

\begin{multline*}
\left\| g_{X}(Y_{1}) - g_{X}(Y_{2})\right\|_{sm} \\
\leq \sup_{0 \leq t \leq 1} \left\| \delta g_{X}(Y_{1} + t(Y_{1} - Y_{2})) \right\|_ { \mathcal{L}(\mathcal{Y}^{sm \times sm})}\left\| Y_{1} - Y_{2}\right\|_{sm}\\
 \leq\sup_{Y \in B_{\rm nc}(Y^{0}, \gamma)_{ sm}} \left\| \delta g_{X}(Y)\right\|_ { \mathcal{L}(\mathcal{Y}^{sm \times sm})}\left\| Y_{1} - Y_{2}\right\|_{ sm}.
\end{multline*}
Thus, by (\ref{eq:est3}), we derive
\begin{equation}\label{eq:est4}
\left\| g_{X}(Y_{1}) - g_{X}(Y_{2})\right\|_{sm} \leq  \frac{1}{2}\left\| Y_{1} - Y_{2}\right\|_{sm}.
\end{equation}
Here we used the convexity of  $B_{\rm nc}(Y^{0}, \gamma)_{sm}$,  so that the segment $(Y_{1}, Y_{2})$ lies in  $B_{\rm nc}(Y^{0}, \gamma)_{sm}$. Hence the mapping $g_{X}$ of the ball  $B_{\rm nc}(Y^{0}, \gamma)_{sm}$ into $\mathcal{Y}^{sm \times sm}$ is contractive with the coefficient $\frac{1}{2}$ for any $X \in B_{\rm nc}(X^{0}, \gamma)_{ sm}.$ However, $g_{X}$ may fail to map $ B_{\rm nc}(Y^{0}, \gamma)_{ sm}$ into itself.

The next step is to find a subset of $B_{\rm nc}(Y^{0},\gamma )_{sm}$  that is a complete metric space mapped by $g_{X}$ into itself for an appropriate choice of $X$.  In fact, we will show that for any $\beta,$ $ 0 < \beta <\gamma,$ there exists a positive $\alpha < $ $\gamma$ such that for any $X \in B_{\rm nc}(X^{0}, \alpha)$, $g_{X} $ maps the closed ball $\closure{B}_{\rm nc}(Y^{0}, \beta)_{sm_{X}}$ into itself. The ball $\closure{B}_{\rm nc}(Y^{0}, \beta)_{sm_{X}}$ as a closed subset of a Banach  space  $\mathcal{Y}^{sm_{X}\times sm_{X}}$ is complete. By assumption 1 of the theorem,  $F(X^{0(m_{X})}, Y^{0(m_{X})}) = $ $F(X^{0}, Y^{0})^{(m_{X})} = 0$, so we can write
\begin{multline}\label{eq:est5}
g_{X}\Big(Y^{0(m_{X})}\Big)\\ = Y^{0(m_{X})} -
\Big(\delta^{Y}F(X^{0(m_{X})}, Y^{0(m_{X})})\Big)^{-1} \Big( F(X,
Y^{0(m_{X})}) - F(X^{0(m_{X})}, Y^{0((m_{X})}) \Big).
\end{multline}
Then
\begin{multline}\label{eq:est6}
\left\|g_{X}(Y) - Y^{0(m_{X})}\right\|_{sm_{X}} \\ \leq \left\| g_{X}(Y) -  g_{X}(Y^{0(m_{X})})\right\|_{sm_{X}} + \left\|g_{X}(Y^{0(m_{X})}) - Y^{0(m_{X})}\right\|_{sm_{X}}.
\end{multline}
By (\ref{eq:est4}),
\begin{equation}\label{eq:est7}
 \left\| g_{X}(Y) -  g_{X}(Y^{0(m_{X})})\right\|_{sm_{X}} \leq \frac{1}{2}\left\| Y - Y ^{0(m_{X})}\right\|_{sm_{X}}.
\end{equation}
 By (\ref{eq:est5}), we can write
\begin{multline}\label{eq:est8}
\left\|g_{X}(Y^{0(m_{X})}) - Y^{0(m_{X})}\right\|_{sm_{X}} \\ \leq
\left\| \Big(\delta^{Y}F(X^{0(m_{X})},
Y^{0(m_{X})})\Big)^{-1}\right\|_{\mathcal{L}(\mathcal{Z}^{sm_{X}
\times sm_{X}}, \mathcal{Y}^{sm_{X} \times sm_{X}})} \\ \times
\left\| F(X, Y^{0(m_{X})}) - F(X^{0(m_{X})},
Y^{0(m_{X})})\right\|_{sm_{X}}.
\end{multline}
By assumption {2} of the theorem, $F$ is continuous at $(X^{0}, Y^{0})$.That is, for any positive  $\beta < \gamma$, there exists a positive $\alpha < \gamma$ such that for any $X \in B_{\rm nc}(X^{0}, \alpha)$ one has
\begin{equation}\label{eq:est9}
\left\| F(X, Y^{0(m_{X})}) - F(X^{0(m_{X})}, Y^{0(m_{X})})\right\|_{sm_{X}}  < \frac{\beta}{2M}.
\end{equation}
Then by (\ref{eq:est1}), (\ref{eq:est8}) and (\ref{eq:est9}),
\begin{equation}\label{eq:est10}
\left\|g_{X}(Y^{0(m_{X})}) - Y^{0(m_{X})}\right\|_{sm_{X}} \leq \frac{1}{2}\beta.
\end{equation}
Thus if $X \in B_{\rm nc}(X^{0}, \alpha)_{sm_{X}}$ and $Y \in  \closure{B}_{\rm nc}(Y^{0}, \beta)_{sm_{X}}$, it follows from (\ref{eq:est6}), (\ref{eq:est7}), and (\ref{eq:est10}) that
\begin{equation*}
\left\|g_{X}(Y) - Y^{0(m_{X})}\right\|_{sm_{X}} < \frac{1}{2}\left\|Y -Y^{0(m_{X})}\right\|_{sm_{X}} + \frac{1}{2}\beta \leq\frac{1}{2}\beta +\frac{1}{2}\beta = \beta .
\end{equation*}
Therefore
\begin{equation}\label{eq:est11}
 g_{X}( \closure{B}_{\rm nc}(Y^{0}, \beta)_{sm_{X}}) \subset \closure{B}_{\rm nc}(Y^{0}, \beta)_{sm_{X}}.
\end{equation}
 In fact, we obtained that \[ g_{X}( \closure{B}_{\rm nc}(Y^{0}, \beta)_{sm_{X}}) \subset B_{\rm nc}(Y^{0}, \beta)_{sm_{X}}.\]
By the classical contraction mapping principle, it follows that for any $X \in B_{\rm nc}(X^{0}, \alpha)$, there exists a unique  $Y = Y_{X }\mathrel {\mathop:} = f(X) \in B_{\rm nc}(Y^{0}, \beta)_{sm_{X}} $ that is a fixed point of the mapping
\[g_{X} \colon \closure{B}_{\rm nc}(Y^{0}, \beta)_{sm_{X}} \rightarrow \closure{B}_{\rm nc}(Y^{0}, \beta)_{sm_{X}}.\]

 Conclusions \textit{I} and  \textit{IIa} follow immediately.
We only need to show that $f$ is a nc function, i.e.,  respects direct sums and similarities.

First we prove that $f$ respects direct sums, i.e., for $X^{\prime}, X^{\prime \prime} \in B_{\rm nc}(X^{0}, \alpha)$, \[f(X^{\prime} \oplus X^{\prime \prime}) = f(X^{\prime}) \oplus f(X^{\prime \prime}).\]
  We have $g_{X^{\prime}}(Y_{X^{\prime}}) = Y_{X^{\prime}} \in B_{\rm nc}(Y^{0}, \beta)_{m_{X^{\prime}}}$,  $g_{X^{\prime \prime}}(Y_{X^{\prime \prime}}) = Y_{X^{\prime \prime}} \in B_{\rm nc}(Y^{0}, \beta)_{m_{X^{\prime \prime}}}$ and
\begin{equation} \label{eq:est12}
g_{X^{\prime} \oplus X^{\prime \prime}}(Y_{X^{\prime} \oplus X^{\prime \prime}}) = Y_{X^{\prime} \oplus X^{\prime \prime}}, \quad Y_{X^{\prime} \oplus X^{\prime \prime}}  \in B_{\rm nc}(Y^{0}, \eta)_{m_{X^{\prime}} + m_{X^{\prime \prime}}}.
\end{equation}
Now, by the definition of $g_{X},$
\begin{multline*}g_{X^{\prime} \oplus X^{\prime \prime}}\Big(Y_{X^{\prime}} \oplus Y_{X^{\prime \prime}}\Big)  \\  =
Y_{X^{\prime}} \oplus Y_{X^{\prime \prime}} - \Big(\delta^{Y}F(X^{0}, Y^{0})\Big)^{-1(m_{X^{\prime}} + m_{X^{\prime \prime}})} F\Big(X^{\prime} \oplus X^{\prime \prime}, Y_{X^{\prime}} \oplus Y_{X^{\prime \prime}}\Big).
\end{multline*}
Since  $F\Big(X^{\prime} \oplus X^{\prime}, Y_{X^{\prime}} \oplus Y_{X^{\prime \prime}}\Big) =  F\Big(X^{\prime}, Y_{X^{\prime}}\Big) \oplus F\Big(X, ^{\prime \prime}, Y_{X^{\prime \prime}}\Big) = 0 \oplus 0 = 0,$ we have
\begin{equation} \label{eq:est13}
g_{X^{\prime} \oplus X^{\prime \prime}}\Big(Y_{X^{\prime}} \oplus Y_{X^{\prime \prime}}\Big) = Y_{X^{\prime}} \oplus Y_{X^{\prime \prime}}.
\end{equation}
Since the fixed point is unique, it follows  from (\ref{eq:est12}) and (\ref{eq:est13})
\[Y_{X^{\prime} \oplus X^{\prime \prime}} =  Y_{X^{\prime}} \oplus Y_{X^{\prime \prime}},\, \, {\rm  or} \quad f\Big(X^{\prime} \oplus X^{\prime \prime}\Big) = f(X^{\prime}) \oplus f(X^{\prime \prime}).\]

 Now we show that $f$ respects similarities. Let  $ X \in B_{\rm nc}(X^{0}, \alpha)$ and let  $S \in \mathbb{F}^{sm_{X} \times sm_{X}}$ be invertible and such that $X^{'} \mathrel {\mathop:} = SXS^{-1}  \in B_{\rm nc}(X^{0}, \alpha).$  Then for any $c \in \mathbb{F},$ we have
\[ \left[
    \begin{array}{cc}
 X^{'}&0
\\ 0 &X
  \end{array}
\right] =
 \left[
    \begin{array}{cc}
I& - c S
\\ 0 &I
  \end{array}
\right] \left[
    \begin{array}{cc}
 X^{'}&0
\\ 0 &X
  \end{array}
\right] \left[
    \begin{array}{cc}
I& c S
\\ 0 &I
  \end{array}
\right],\]
or equivalently, $ X^{'} \oplus X = T_{c}\Big(X^{'} \oplus X\Big)T_{c}^{-1},$ where
\[ T_{c} \mathrel {\mathop:} =
\left[
    \begin{array}{cc}
I& - c S
\\ 0 &I
  \end{array}
\right].
\]
For $c \neq 0$ small enough, we have $T_{c}f\Big(X^{'} \oplus X \Big)T_{c}^{-1} \in B_{\rm nc}(Y^{0}, \beta).$ Then \[\Big( X^{'} \oplus X,  T_{c}f\Big(X^{'} \oplus X\Big)T_{c}^{-1}\Big) \subset \Omega.\]  Since
 \begin{multline*}
F\Big( X^{'} \oplus X, T_{c}f\Big(X^{'} \oplus X\Big)T_{c}^{-1}\Big) =
 F\Big(T_{c}\Big( X^{'} \oplus X\Big)T_{c}^{-1}, T_{c}f\Big(X^{'} \oplus X\Big)T_{c}^{-1}\Big)\\
 = T_{c}F\Big( X^{'} \oplus X, f\Big(X^{'} \oplus X\Big)\Big)T_{c}^{-1} = 0,
\end{multline*}
by the part proved earlier, we must have $T_{c}f(X^{'} \oplus X)T_{c}^{-1} = f(X^{'} \oplus X).$ Using $f(X^{'} \oplus X) = f(X^{'}) \oplus f(X),$ we obtain
\[
T_{c}\Big(f(X^{'}) \oplus f(X)\Big)T_{c}^{-1} = f(X^{'}) \oplus f(X),
\]
or equivalently,
\[
 \left[
    \begin{array}{cc}
I& - c S
\\ 0 &I
  \end{array}
\right] \left[
    \begin{array}{cc}
 f(X^{'})&0
\\ 0 &f(X)
  \end{array}
\right] \left[
    \begin{array}{cc}
I& c S
\\ 0 &I
  \end{array}
\right] =\left[
    \begin{array}{cc}
f( X^{'})&0
\\ 0 &f(X)
  \end{array}
\right].
\]
The left-hand side is equal to
\[
\left[
    \begin{array}{cc}
 f(X^{'})& c(f(X^{'})S - Sf(X))
\\ 0 &f(X)
  \end{array}
\right].
\]
Hence $f(X^{'})S = Sf(X),$ i.e., $f(SXS^{-1}) = f(X^{'}) = Sf(X)S^{-1}$ as required.

 To prove conclusion \textit{IIb}, we observe that   $f( B_{\rm nc}(X^{0}, \alpha)) \subset \closure{B}_{\rm nc}(Y^{0}, \beta)$  implies $\|f(X)\| \leq \|Y^{0}\| +  \beta$ for any $X \in B_{\rm nc}(X^{0}, \alpha)$. By \cite[Corollary 7.28]{KV-VV}, $f$ is bounded on $B_{\rm nc}(X^{0}, \alpha)$ if and only if  $f$ is uniformly analytic on $B_{\rm nc}(X^{0}, \alpha).$

Next we prove conclusion \textit{III}.
We have that
$F(X, f(X)) = 0$ for every $X \in B_{\rm nc}(X^{0}, \alpha).$ By \cite[Theorem 7.51 and Theorem 7.53]{KV-VV}, the derivatives of $F$ exist and moreover they are continuous on
\[B_{\rm nc}(X^{0}, \alpha) \times B_{\rm nc}(Y^{0}, \beta) \subset B_{\rm nc}((X^{0}, Y^{0}), \delta).\]
In particular,   $\delta^{Y}F$ and $\delta^{X}F$ are  continuous,
$$\delta^{Y}F(X, Y) \in \mathcal{L}_{\rm
cb}(\mathcal{Y}^{sm_{(X,Y)} \times sm_{(X,Y)}},
\mathcal{Z}^{sm_{(X,Y)} \times sm_{(X,Y)}})$$  and
$$\delta^{X}F(X, Y)\in\mathcal{L}_{\rm cb}(\mathcal{X}^{sm_{(X,Y)}
\times sm_{(X,Y)}}, \mathcal{Z}^{sm_{(X,Y)} \times sm_{(X,Y)}})$$
for every $(X, Y) \in B_{\rm nc}(X^{0}, \alpha) \times B_{\rm
nc}(Y^{0}, \beta).$  Using the chain rule we obtain
\[\delta^{X}F(X,f(X)) + \delta^{Y}F(X, F(X))(\delta f(X)) = 0.\]
Making $\beta$ (and hence the corresponding $\alpha$) smaller if necessary, we can make $\delta^{Y}F(X, Y)$ invertible on $B_{\rm nc}(X^{0}, \alpha) \times B_{\rm nc}(Y^{0}, \beta)$, with a completely bounded inverse. Then
\[\delta f(X) = - \Big(\delta^{Y}F(X, f(X))\Big)^{-1}\delta^{X}F(X, f(X))\]
as required.

\end{proof}

\begin{proof}[Proof of Theorem \ref{thm:ops-2}] Let \[\Omega  \mathrel {\mathop:} =  \coprod_{n =1}^{\infty}\left ( \mathcal{X}^{n \times n} \times \Omega_{n}^{Y}\right ) \subset (\mathcal{X} \times \mathcal{Y})_{\rm nc}.\]

 Consider the function  $F\colon \Omega  \rightarrow {\mathcal{X}}_{\rm nc}$,  $F(X, Y) = g(Y) - X$.
Denote $X^{0}= g(Y^{0}).$  By the construction of $F$, $F(X^{0}, Y^{0}) = g(Y^{0}) - X^{0} = 0,$
 and $F$ is uniformly analytic on $\Omega.$ $\delta^{Y}F(X^{0}, Y^{0}) = \delta g(Y^{0})$ is invertible,
  and $(\delta^{Y}F(X^{0}, Y^{0}))^{-1} = (\delta g(Y^{0}))^{-1} \in \mathcal{L}_{\rm cb}(\mathcal{X}^{s \times s}, \mathcal{Y}^{s \times s}).$ Thus all the assumptions of Theorem \ref{thm:ops-1} hold for $F$,  and  we obtain the following conclusions:
\begin{itemize}
\item There exist $\alpha >0,$ $\beta >0$ such that for any $m \in \mathbb{N},$
\[ B_{\rm nc}(X^{0}, \alpha)_{sm} \times B_{\rm nc}(Y^{0}, \beta)_{sm} \subset \Omega_{sm}.\]
\item There exists a nc function $ f \colon B_{\rm nc}(X^{0}, \alpha) \rightarrow B_{\rm nc}(Y^{0}, \beta),$
such that  for $(X, Y) \in  B_{\rm nc}(X^{0}, \alpha)_{sm} \times B_{\rm nc}(Y^{0}, \beta)_{sm},$
\[(F(X, Y) =0) \iff  (Y = f(X)).\]
\item $\delta^{Y}F(X, Y)$ has a completely bounded inverse on
$B_{\rm nc}(X^{0}, \alpha) \times B_{\rm nc}(Y^{0}, \beta)$ for every $m\in\mathbb{N}$.
\item $f$ is uniformly analytic on $B_{\rm nc}(X^{0}, \alpha)$  with
\[\delta f(X) = -\Big(\delta^{Y}F(X, f(X))\Big)^{-1} \delta^{X}F(X,f(X)).\]

\end{itemize}

 We now prove that $f$ is the inverse function for $g$, and that $\delta f(X) = \Big(\delta g(f(X))\Big)^{-1}.$
Since $F(X, Y) = g(Y) - X,$
\[ F(X, f(X)) = 0 \Rightarrow  F(X, f(X)) = g(f(X)) -X = 0 \Rightarrow  g(f(X)) = X.\]
 Next, we have
\begin{multline*}
\delta f(X) = -\Big(\delta^{Y}F(X, f(X))\Big)^{-1} \delta^{X}F(X,f(X)) \\ = -\Big(\delta^{Y}F(X, f(X))\Big)^{-1}(-{\rm id}_{{\mathcal{X}}^{sm_{X}\times sm_{X}}}) = \Big(\delta g(f(X))\Big)^{-1}.
\end{multline*}

Set  $\Gamma = B_{\rm nc}(X^{0}, \alpha),$ $ \Delta = f(\Gamma) \subset B_{\rm nc}(Y^{0}, \beta).$ Clearly,  $f\colon \Gamma \rightarrow \Delta$ is a surjection. Since $g(f(X)) = X$ for every $X \in \Gamma, f$ is also an injection. And since $f$ is uniformly analytic,  $ f$  is continuous. We also have that $f^{-1} = g|_{\Delta}$ is continuous. Thus $g|_{\Delta }\colon \Delta \rightarrow \Gamma$ is a homeomorphism, and $\Delta$ is a uniformly open nc set.
\end{proof}

\begin{lemma} \label{thm:nilp2}
In the assumptions of Theorem \ref{thm:nilp1}, for any $(X^{1}, Y^{1}), $ $(X^{2}, Y^{2}) \in \Nilp(\mathcal{M} \times \mathcal{N}, (X^{0}, Y^{0})),$
the operator $ \Delta_{R}^{Y}F((X^{1}, Y^{1}), (X^{2}, Y^{2})) $ is invertible.
\end{lemma}
\begin{proof}
First we observe that for any $(X^{1}, Y^{1}),$ $ (X^{2}, Y^{2}) \in \Nilp(\mathcal{M} \times \mathcal{N}, (X^{0}, Y^{0})),$
$\left[
    \begin{smallmatrix}
X^{1}& 0
\\ 0 &X^{2}
  \end{smallmatrix}
\right]$ and $   \left[
    \begin{smallmatrix}
Y^{1}& Z
\\ 0 &Y^{2}
  \end{smallmatrix}\right]
$ are jointly nilpotent about $(X^{0}, Y^{0}).$
\begin{multline*}
\Delta_{R}^{Y}F((X^{1}, Y^{1}), (X^{2}, Y^{2}))(Z) = F\left(   \left[
    \begin{array}{cc}
X^{1}& 0
\\ 0 &X^{2}
  \end{array} \right],  \left[
    \begin{array}{cc}
Y^{1}& Z
\\ 0 &Y^{2}
  \end{array}
\right]\right)_{(1, 2)} \\ =
\sum_{\ell = 1}^{\infty} \left\{\left(   \left[
    \begin{array}{cc}
X^{1} - X^{0(m)}& 0
\\ 0 &X^{2} - X^{0(m)}
  \end{array}
\right],  \left[
    \begin{array}{cc}
Y^{1} - Y^{0(m)}& Z
\\ 0 &Y^{2} - Y^{0(m)}
  \end{array}
\right]\right)^{\odot_{s} \ell}\right\}_{(1, 2)} \\ \times
\Delta_{R}^{\ell}F((X^{0}, Y^{0}), (X^{0}, Y^{0}))\\ =
Z\Delta_{R}^{Y}F((X^{0}, Y^{0}), (X^{0}, Y^{0})) + \sum_{\ell = 1}^{\infty}\sum_{j = 0}^{\ell - 1}(X^{1} - X^{0(m)}, Y^{1} - Y^{0(m)})^{\odot_{s} j}\odot_{s} Z \\  \odot_{s}
(X^{2} - X^{0(m)}, Y^{2} - Y^{0(m)})^{\odot_{s} (\ell - j)}
\Delta^{\ell + 1}_{R}F((X^{0}, Y^{0}), (X^{0}, Y^{0}))  \\ =
\Big( \Delta_{R}^{Y}F((X^{0}, Y^{0}), (X^{0}, Y^{0}))\Big)^{(m)}\Big(({\rm id} + N)(Z)\Big),
\end{multline*}
where
\begin{multline*}
N(Z) =  \sum_{\ell = 1}^{\infty}\sum_{j = 0}^{\ell - 1}(X^{1} - X^{0(m)}, Y^{1} - Y^{0(m)})^{\odot_{s} j}\odot_{s} Z  \odot_{s}
(X^{2} - X^{0(m)}, Y^{2} - Y^{0(m)})^{\odot_{s}(\ell - j)} \\
\times \Big(\Delta_{R}^{Y}F((X^{0}, Y^{0}), (X^{0}, Y^{0}))\Big)^{-1}\Delta^{\ell + 1}_{R}F((X^{0}, Y^{0}), (X^{0}, Y^{0})).
\end{multline*}
The linear operator $N$ is nilpotent, i.e., $N^{\gamma} = 0$ for some $\gamma \in \mathbb{N}.$ Therefore  the operator $  {\rm id} + N$ is invertible.  Thus  the operator $\Delta_{R}^{Y}F((X^{1}, Y^{1}), (X^{2}, Y^{2}))$
is invertible.
\end{proof}

\begin{proof}[Proof of Theorem \ref{thm:nilp1}]
 For  every $m, \kappa \in \mathbb{N}$ and $ X \in \Nilp(\mathcal{M}, X^{0}; sm, \kappa),$  we define  the set
\[\Upsilon_{X} \mathrel {\mathop:} =  \left \{ Y \in \mathcal{N}^{sm \times sm} \colon (X, Y) \in \Nilp(\mathcal{M} \times \mathcal{N}, (X^{0},Y^{0}), \kappa^{\prime}), {\rm for \,\, some\,\, } \kappa^{\prime} \geq \kappa \right \}.\] Notice that $\Upsilon_{X} \subseteq\Nilp(\mathcal{N}, Y^{0}, \kappa^{\prime})$ for some  $\kappa^{\prime} \geq \kappa$. We also define the mapping $g_{X}\colon \Upsilon_{X}\rightarrow \Upsilon_{X}$ by
\begin{equation}\label{eq:nilp-2}
 g_{X}(Y) = Y - \Big(\Delta_{R}^{Y}F((X^{0}, Y^{0}), (X^{0}, Y^{0}))\Big)^{-1(m)} F(X, Y).
\end{equation}
We now show that $g_{X}$ is well defined, i.e., $ g_{X}(Y) \in
\Upsilon_{X}$ for $Y \in \Upsilon_{X}.$ Indeed, since $Y \in
\Upsilon_{X}$ implies that $(X, Y) \in \Nilp(\mathcal{M} \times
\mathcal{N}, (X^{0}, Y^{0}); \kappa^{\prime})$  for some
$\kappa^{\prime} \geq \kappa,$ it suffices to show that
 \[
\Big(X - X^{0(m)}, g_{X}(Y) - Y^{0(m)}\Big)^{\odot_{s} \kappa^{\prime}} = 0.
\]
 Using the TT formula
\begin{multline*}
 F(X,Y) = F(X^{0(m)}, Y^{0(m)})\\
+ \sum_{\ell =1}^{\kappa^{\prime} -1} (X - X^{0(m)}, Y - Y^{0(m)})^{\odot_{s} \ell} \Delta_{R}^{\ell}F((X^{0}, Y^{0}), (X^{0}, Y^{0})),
\end{multline*}
we obtain
\begin{multline*}
\Big( X - X^{0(m)}, g_{X}(Y) - Y^{0(m)}\Big)^{\odot_{s} \kappa^{\prime}} \\ = \Big( X - X^{0(m)}, Y -  Y^{0(m)} -  \Big(\Delta_{R}^{Y}F((X^{0(m)}, Y^{0(m)}), (X^{0(m)}, Y^{0(m)}))\Big)^{-1}\\
\times \sum_{\ell =1}^{\kappa^{\prime} -1} (X - X^{0(m)}, Y - Y^{0(m)})^{\odot_{s} \ell }
\Delta_{R}^{\ell}F((X^{0}, Y^{0}), (X^{0}, Y^{0}))\Big)^{\odot_{s }\kappa^{\prime}}  \\
= \Big( X - X^{0(m)}, Y -  Y^{0(m)} -
 \sum_{\ell =1}^{\kappa^{\prime} -1} (X - X^{0(m)}, Y - Y^{0(m)})^{\odot_{s} \ell }\\
\times \Big(\Delta_{R}^{Y}F((X^{0}, Y^{0}), (X^{0}, Y^{0}))\Big)^{-1} \Delta_{R}^{\ell}F((X^{0}, Y^{0}), (X^{0}, Y^{0})) \Big)^{\odot_{s }\kappa^{\prime}}= 0,
\end{multline*}
since
\[
( X - X^{0(m)}, Y - Y^{0(m)})^{\odot_{s} w} = 0
\]
for every word $w$ in two letters $g_{1},$ $g_{2}$ of length $\kappa^{\prime}$ or greater.
 Here for $w = g_{i_{1}} \hdots g_{i_{k}}$ we define
$$
\Big( X - X^{0(m)}, Y - Y^{0(m)}\Big)^{\odot_{s} w} = \Big(
X_{i_{1}} - X^{0(m)}_{i_{1}}\Big){\odot_{s}} \cdots
{\odot_{s}}\Big( X_{i_{k}} - X^{0(m)}_{i_{k}}\Big),
$$
where
\begin{equation*}
X_{i} - X^{0(m)}_{i}  \mathrel {\mathop:} =  \begin{cases}X - X^{0(m)}, & \mbox{if } i\mbox{ = 1} \\ Y - Y^{0(m)}, & \mbox{if } i\mbox{ = 2} .\end{cases}
\end{equation*}
The TT formula also implies that  $F(X, Y) \in \Nilp(\mathcal{O}; sm, \kappa^{\prime})$ for $(X, Y) \in \Nilp(\mathcal{M}\times \mathcal{N}, (X^{0}, Y^{0}); sm, \kappa^{\prime}).$

Now we prove that for every  $ m,\kappa \in \mathbb{N}$ and  $X \in \Nilp(\mathcal{M}, X^{0}; sm,\kappa),$ there exists  $Y  \in \Upsilon_{X}$  such that $F(X, Y) = 0.$ We define  a sequence
\[  Y^{[0]} = Y^{0(m)},  \, \, Y^{[k + 1]} = g_{X}(Y^{[k]}), \, \, k = 0, 1, \hdots, \]
and  claim that
\begin{equation}\label{eq:nilp-3}
Y^{[k + 1]} - Y^{[k]} = (X - X^{0(m)})^{\odot_{s} k+1} f_{k + 1} +  {\rm higher  \,\, order \,\, terms},
\end{equation}
where $f_{k+1} \in \Hom ((\mathcal{M}^{s \times s})^{\otimes k + 1}, \mathcal{N}^{s \times s})$, and is extended to \[\Hom \Big(((\mathcal{M}^{s \times s})^{\otimes k + 1})^{m \times m}, (\mathcal{N}^{s \times s})^{m \times m}\Big)\] by $f_{k + 1} (A) = \left [ f_{k + 1}(a_{ij})\right ]_{i,j = 1, \hdots , m}.$

We apply  induction on $k$.  Note  that $(X, Y^{0(m)})$ is in $\Nilp(\mathcal{M} \times \mathcal{N}, (X^{0}, Y^{0}); sm, \kappa)$ since $(X - X^{0(m)}, Y^{0(m)}-  Y^{0(m)}) = (X - X^{0(m)}, 0).$  We first show that  (\ref{eq:nilp-3}) holds for $k = 0$. Using (\ref{eq:nilp-2})  and the TT formula for $F( \,\cdot \,, Y^{0(m)}),$ we obtain
\begin{multline}\label{eq:nilp-4}
Y^{[1]} - Y^{[0]} = -  \Big(\Delta_{R}^{Y}F((X^{0(m)}, Y^{0(m)}), (X^{0(m)}, Y^{0(m)}))\Big)^{-1} F(X, Y^{0(m)})\\
=- \sum^{\kappa -1}_{\ell = 1}(X - X^{0(m)})^{\odot_{s} \ell}\Big(\Delta_{R}^{Y} F((X^{0}, Y^{0}), (X^{0}, Y^{0}))\Big)^{-1} \Big (\Delta_{R}^{X}\Big)^{ \ell}F((X^{0}, Y^{0}), (X^{0}, Y^{0}))\\
= - (X - X^{0(m)}) \Big(\Delta_{R}^{Y} F((X^{0}, Y^{0}), (X^{0}, Y^{0}))\Big)^{-1}\Delta_{R}^{X}F((X^{0}, Y^{0}), (X^{0}, Y^{0})) \\ +   {\rm higher  \,\, order \,\, terms} \\ =
(X - X^{0(m)}) f_{1} +  {\rm higher  \,\, order \,\, terms}.
\end{multline}

Let (\ref{eq:nilp-3}) be true for all powers $0, 1, \hdots, k - 1,$ where  $k \in \mathbb{N}.$ Using (\ref{eq:nilp-2})  and the TT formula for $F( X, Y^{[k]}),$ we obtain
\begin{multline*}
Y^{[k + 1]} - Y^{[k]} = - \Big(\Delta_{R}^{Y}F((X^{0(m)}, Y^{0(m)}), (X^{0(m)}, Y^{0(m)}))\Big)^{-1}F(X, Y^{[k]})\\
=- \Big(\Delta_{R}^{Y}F((X^{0(m)}, Y^{0(m)}), (X^{0(m)}, Y^{0(m)}))\Big)^{-1}\\
\times \sum^{\infty}_{\ell = 1}(X - X^{0(m)}, Y^{[k]} - Y^{0(m)})^{\odot_{s} \ell}
 \Delta_R^\ell F((X^{0}, Y^{0}), (X^{0}, Y^{0}))
\\= - \sum^{\infty}_{\ell = 1}(X - X^{0(m)}, Y^{[k]} - Y^{[k - 1]} + Y^{[k - 1]} - Y^{0(m)})^{\odot_{s} \ell} \\ \times \Big(\Delta_{R}^{Y} F((X^{0}, Y^{0}), (X^{0}, Y^{0}))\Big)^{-1}
 \Delta_R^\ell F((X^{0}, Y^{0}), (X^{0}, Y^{0}))  \\ = {\rm  summands \,\,with \,\, powers \,\, not \, \, involving} \,\, (Y^{[k]} - Y^{[k - 1]}) \\
+  {\rm  summands \,\,with \,\, powers \,\, involving} \,\, (Y^{[k]} - Y^{[k - 1]})\\
= - \sum^{\infty}_{\ell = 1}(X - X^{0(m)}, Y^{[k - 1]} - Y^{0(m)})^{\odot_{s} \ell} \\ \times \Big (\Delta_{R}^{Y} F((X^{0}, Y^{0}), (X^{0}, Y^{0}))\Big)^{-1}  \Delta_R^{\ell}F((X^{0}, Y^{0}), (X^{0}, Y^{0}))   \\ -
(Y^{[k]} - Y^{[k - 1]}) +
{\rm  higher \,\,order \,\, terms \,\, } \\
= - \Big(\Delta_{R}^{Y} F((X^{0(m)}, Y^{0(m)}), (X^{0(m)}, Y^{0(m)}))\Big)^{-1}F(X, Y^{[k - 1]}) - (Y^{[k]} - Y^{[k - 1]})  \\+
 (X - X^{0(m)})^{\odot_{s} k + 1} f_{k + 1} +  {\rm higher  \,\, order \,\, terms} \\ =
(X - X^{0(m)})^{\odot_{s} k + 1} f_{k + 1} +  {\rm higher  \,\, order \,\, terms} .
\end{multline*}
Note that all  the sums have finitely many nonzero terms.

It follows from (\ref{eq:nilp-3}) that  $Y^{[k + 1]} - Y^{[k]} \in  \Nilp\Big(\mathcal{N}; sm, \lceil \frac{\kappa}{k + 1} \rceil\Big).$ In particular,  $Y^{[\kappa]} - Y^{[\kappa - 1]} \in  \Nilp(\mathcal{N}; sm, 1),$ i.e.,  $  Y^{[\kappa]} - Y^{[\kappa - 1]} = 0,$ and thus $F(X, Y^{[\kappa - 1]}) = 0.$

Now define $f(X)\mathrel {\mathop:} =  Y^{[\kappa - 1]}$. Obviously, $(X, Y^{[\kappa -1]}) \in \Nilp(\mathcal{M} \times \mathcal{N}, (X^{0}, Y^{0}); sm)$ so $ f( \Nilp(\mathcal{M}, X^{0}; sm)) \subset \Nilp(\mathcal{N}, Y^{0}; sm).$ Thus,  for each $X \in \Nilp(\mathcal{M}, X^{0})$ there exists $Y(= f(X)) \in {\Upsilon}_{X}$ such that $F(X, Y) = 0.$

Next we  show that the solution of $F(X, Y) = 0$ for  each $X \in \Nilp(\mathcal{M}, X^{0})$ is unique. Moreover, if  $ (X, Y), (X, Y^{\prime}) \in \Nilp(\mathcal{M} \times \mathcal{N}, (X^{0}, Y^{0}); sm)$  are such that $F(X, Y) = F(X, Y^{\prime}),$ then we have $Y = Y^{\prime}.$
By \cite[Theorem 2.10]{KV-VV}) (cf. \eqref{eq:intro6})
\[ 0=F(X, Y) - F(X, Y') = \Delta_{R}^{Y}F((X, Y'), (X, Y))(Y - Y').
\]
By  Lemma \ref{thm:nilp2}, the operator $\Delta_{R}^{Y}F((X, Y'), (X, Y))$ is invertible, hence $Y = Y^{\prime}.$

 We have proved that for each $X$ $ \in \Nilp(\mathcal{M}, X^{0})$ there exists a unique $Y(= f(X))$ $ \in \Upsilon_{X}$ such that $F(X, Y) = 0.$

We prove now that $f \colon \Nilp(\mathcal{M}, X^{0}) \rightarrow \Nilp(\mathcal{N}, Y^{0})$ is a nc function. Let $X  \in \Nilp(\mathcal{M}, X^{0}; sm),$ $\tilde{X}  \in \Nilp(\mathcal{M}, X^{0}; s\tilde{m}),$ and  $S \in \mathcal{M}^{s\tilde{m} \times sm}$ be such that $SX = \tilde{X}S.$ Then we have $(X, f(X)) \in \Nilp(\mathcal{M} \times \mathcal{N}, (X^{0}, Y^{0}); sm),$  $(\tilde{X}, f(\tilde{X})) \in \Nilp(\mathcal{M} \times \mathcal{N}, (X^{0}, Y^{0}); s\tilde{m}).$ Applying \cite[Theorem 2.11]{KV-VV}, we obtain
\begin{multline*}
SF(X, f(X)) - F(\tilde{X}, f(\tilde{X}))S \\ = \Delta_{R} F((\tilde{X}, f(\tilde{X})), (X, f(X))
(S(X, f(X))  - (\tilde{X}, f(\tilde{X}))S)  \\ =
\Delta_{R}^{X} F((\tilde{X}, f(\tilde{X})), (X, f(X))
(SX  - \tilde{X}S) \\ +
\Delta_{R}^{Y} F((\tilde{X}, f(\tilde{X})), (X, f(X))
(Sf(X)  - f(\tilde{X})S)\\
= \Delta_{R}^{Y} F((\tilde{X}, f(\tilde{X})), (X, f(X))
(Sf(X)  - f(\tilde{X})S) .
\end{multline*}
Since the left-hand side is $0$ and by Lemma  \ref{thm:nilp2} the operator $$\Delta_{R}^{Y} F((\tilde{X}, f(\tilde{X})), (X, f(X)) $$ is invertible, we obtain that   $Sf(X)  -  f(\tilde{X})S = 0,$ i.e., $f $ respects intertwinings. Thus we have proved conclusion \textit{I}.

Finally, we prove conclusion  \textit{II}.   If $X  \in \Nilp(\mathcal{M}, X^{0}; sm, \kappa),$ then  $(X, f(X)) \in \Nilp(\mathcal{M} \times \mathcal{N}, (X^{0}, Y^{0}); sm, \kappa).$  We also observe that
$\left[
    \begin{smallmatrix}
X& Z
\\ 0 &X
  \end{smallmatrix}
\right]$ $\in \Nilp(\mathcal{M}, X^{0}; 2sm, \kappa+1) $ for an arbitrary $Z \in \mathcal{M}^{sm \times sm}.$ Hence
\begin{multline*}
\left(   \left[
    \begin{array}{cc}
X& Z
\\ 0 &X
  \end{array} \right],  f\left( \left[
    \begin{array}{cc}
X& Z
\\ 0 &X
  \end{array}
\right]\right)\right) \\
= \left(   \left[
    \begin{array}{cc}
X& Z
\\ 0 &X
  \end{array} \right],  \left[
    \begin{array}{cc}
f(X)& \Delta_{R} f(X, X)(Z)
\\ 0 &f(X)
  \end{array}
\right]\right) \in \Nilp(\mathcal{M} \times \mathcal{N} , (X^{0}, Y^{0}); 2sm, \kappa+1).
\end{multline*}
By part \textit{I},
\begin{equation*}
0 =  F\left(   \left[
    \begin{array}{cc}
X& Z
\\ 0 &X
  \end{array} \right],  \left[
    \begin{array}{cc}
f(X)& \Delta_{R} f(X, X)(Z)
\\ 0 &f(X)
  \end{array}
\right]\right).
\end{equation*}
Therefore
\begin{multline*}
0 =  F\left(   \left[
    \begin{array}{cc}
X& Z
\\ 0 &X
  \end{array} \right],  \left[
    \begin{array}{cc}
f(X)& \Delta_{R} f(X, X)(Z)
\\ 0 &f(X)
  \end{array}
\right]\right)_{(1, 2)} \\
= \Delta_{R}F\Big((X, f(X)), (X,f(X))\Big) \Big(Z, \Delta_{R} f(X, X)(Z)\Big ) \\
= \Delta_{R}^{X}F\Big((X, f(X)), (X,f(X))\Big)(Z)  +\Delta_{R}^{Y}F\Big((X, f(X)), (X,f(X))\Big) (\Delta_{R}f(X, X)(Z)).
\end{multline*}
Since $\Delta_{R}^{Y}F\Big((X, f(X)), (X,f(X))\Big)$ is invertible by Lemma \ref{thm:nilp2}, we obtain
\begin{multline*}
\Delta_{R}f(X, X)(Z) =\\
 - \Big(\Delta_{R}^{Y}F((X, f(X)), (X,f(X)))\Big)^{-1} \Delta_{R}^{X}F\Big((X, f(X)), (X,f(X))\Big)(Z).
\end{multline*}
\end{proof}

\begin{proof}[Proof of Theorem \ref{thm:nilp3}]
Denote $X^{0} = g(Y^{0}).$ Consider the function $F(X, Y) = g(Y) - X.$ Then by the construction, we have
\begin{itemize}
\item $ F\colon \Nilp(\mathcal{M} \times \mathcal{N}, (X^{0}, Y^{0})) \rightarrow \mathcal{M}_{\rm nc},$
\item $F(X^{0}, Y^{0}) = 0,$
\item $F $ is a nc function, and $\Delta_{R}^{Y}F((X^{0}, Y^{0}), (X^{0}, Y^{0}) ) = \Delta_{R}g(Y^{0}, Y^{0})$ is invertible.
\end{itemize}
Thus all the assumptions of  Theorem \ref{thm:nilp1} hold, and  we obtain the following conclusions:
\begin{itemize}
\item  There exists a nc function $ f\colon \Nilp(\mathcal{M}, X^{0}) \rightarrow \Nilp(\mathcal{N}, Y^{0}) \subset \mathcal{N}_{\rm nc}$  such that $X \in \Nilp(\mathcal{M}, X^{0}) $ if and only if  $(X, f(X)) \in \Nilp(\mathcal{M} \times \mathcal{N}, (X^{0}, Y^{0})),$ and $Y = f(X)$ if and only if  $F(X, Y) = 0.$

\item  For any $X \in \Nilp(\mathcal{M}, X^{0}),$ the operator $\Delta_{R}^{Y}F((X, f(X)), (X, f(X)))$ is invertible, and \begin{multline} \label{eq:nilp-7}
\Delta_{R}f(X, X)(Z) \\
= -\Big(\Delta_{R}^{Y}F((X, f(X)), (X, f(X)))\Big)^{-1} \Delta_{R}^{X}F((X, f(X)), (X, f(X)))(Z).
\end{multline}
\end{itemize}
If $g(Y^{1}) = g(Y^{2}),$ then $g(Y^{1}) - g(Y^{2}) = \Delta_{R}g(Y^{2}, Y^{1})(Y^{1} - Y^{2})$ for any  $Y^{1}$  and $ Y^{2} \in \Nilp(\mathcal{N}, Y^{0}; sm)$ for some $m.$  By Lemma \ref{thm:nilp2}, the operator $\Delta_{R}g(Y^{2}, Y^{1}) = \Delta_{R}^{Y}F\Big((X^{0(m)}, Y^{2}), (X^{0(m)}, Y^{1})\Big)$  is invertible. Therefore, $ Y^{1} = Y^{2}.$ This means that $g$ is one-to-one.

Now from  $0 = F(X,Y) = g(Y) - X,$ we have  $g\Big(f(X)\Big) \equiv X$ for every $X \in \Nilp(\mathcal{M}, X^{0}).$  Hence $g\Big(\Nilp(\mathcal{N}, Y^{0})\Big) \supseteq \Nilp(\mathcal{M}, X^{0}).$ By \cite[Remark 5.7]{KV-VV}, we also have the inclusion ``$\subseteq$", so that $g\Big(\Nilp(\mathcal{N}, Y^{0})\Big) = \Nilp(\mathcal{M}, X^{0}).$

We have  $g\Big(f(X)\Big) = X$ for every $X \in \Nilp(\mathcal{M}, X^{0}).$  Also, given $Y \in \Nilp(\mathcal{N}, Y^{0})$, we have $g(f(g(Y))) = g(Y).$ Since $g$ is one-to-one, it follows that $Y = f\Big(g(Y)\Big).$ Thus $f  =g^{-1}.$

By Lemma \ref{thm:nilp2}, $\Delta_{R}g(Y, Y) = \Delta_{R}^{Y}F\Big((X, Y), (X, Y)\Big)$ is invertible, where $X\in\Nilp(\mathcal{M}, X^0)$ is arbitrary.
Therefore $\Delta_Rg(f(X),f(X))$ is invertible for every $X\in\Nilp(\mathcal{M}, X^0)$, and
\begin{multline*}
\Delta_{R}f(X, X) = - \Big(\Delta_{R}^{Y}F((X, f(X)), (X, f(X)))\Big)^{-1}\Delta_{R}^{X}F\Big((X, f(X)), (X, f(X))\Big) \\ =
- \Big(\Delta_{R}^YF((X, f(X)), (X, f(X))\Big)^{-1}(- {\rm id}) = \Big(\Delta_{R}^YF((X, f(X)), (X, f(X))\Big)^{-1}\\
=\Big(\Delta_Rg(f(X),f(X))\Big)^{-1}.
\end{multline*}
The proof is complete.
\end{proof}

\section{Applications}

\subsection{Initial value problems for ODEs in nc spaces}
In our earlier paper \cite{AK}, we obtained a nc version of the
Banach contraction mapping theorem and then applied it to obtain a
theorem on the existence and uniqueness of the solution of the
initial value problem for ODEs of the form $\dot{Y}=g(t,Y)$ in nc
spaces. The right-hand side $g$ was assumed to be a nc function of
$Y$ for every fixed $t\in\mathbb{R}$ and satisfied a global
Lipschitz condition. Consequently, the solution is also globally
defined and having a certain direct-sum structure when the initial
condition has a similar structure. In this section, we obtain a
complementary result using the implicit nc function theorem. The
global Lipschitz condition on $g$ is now replaced by the
assumption of continuity of its Gateaux derivative. Then the
existence and uniqueness of the local solution of the initial
value problem is established. Moreover, we show that the solution
is a uniformly analytic nc function of the initial data.

\begin{thma}
Let $\mathcal{I}$ be an interval in $\mathbb{R}$ and $t_0$ an
interior point of $\mathcal{I}$. Let $\mathcal{X}$ be a (real or
complex) operator space, $\Gamma\subseteq\mathcal{X}_{\rm nc}$ a
uniformly open nc set, $s\in\mathbb{N}$, and $X^0\in\Gamma_s$.
Suppose that $g\colon\mathcal{I}\times\Gamma \to\mathcal{X}_{\rm
nc}$ is a continuous mapping with respect to the uniformly-open
topology on $\mathcal{X}_{\rm nc}$, its G-derivative $\delta g$ is
continuous in the norm $\|\cdot\|_{\mathcal{L}_{\rm
cb}(\mat{\mathcal{X}}{s})}$, and that
$g(t,\cdot)\colon\Gamma\to\mathcal{X}_{\rm nc}$ is a nc function
for every fixed $t\in\mathcal{I}$. Then there exist $\delta>0$
with $[t_0-\delta,t_0+\delta]\subset\mathcal{I}$, $\alpha>0$, and
a uniformly analytic nc function
\begin{equation}\label{eq:uaf}
f\colon B_{\rm nc}(X^0,\alpha)\to
C^1([t_0-\delta,t_0+\delta],\mat{\mathcal{X}}{s})_{\rm nc}
\end{equation}
such that, for every $X\in B_{\rm nc}(X^0,\alpha)$, $Y=f(X)$ is a
unique solution of the initial value problem for the ODE
\begin{equation}\label{eq:IVP}
\dot{Y}=g(t,Y),\qquad Y(t_0)=X.
\end{equation}
Here $C^1([t_0-\delta,t_0+\delta],\mat{\mathcal{X}}{s})$ is an
operator space of $\mat{\mathcal{X}}{s}$-valued continuously
differentiable functions on $[t_0-\delta,t_0+\delta]$ with respect
to the sequence of norms
\begin{equation*}
  \|Y\|_{sm}=\max\{\|Y\|_\infty,\|\dot{Y}\|_\infty\},\quad
  m=1,2,\ldots,
\end{equation*}
where
$$\|Y\|_\infty=\max_{t\in[t_0-\delta,t_0+\delta]}\|Y(t)\|_{sm}$$
and similarly for $\|\dot{Y}\|_\infty$.
\end{thma}
\begin{proof}
By the Cauchy theorem on the existence and uniqueness of the
solution of the initial value problem for an ODE (see, e.g.,
\cite[Theorem 10.4.5]{D}), for $X=X^0$, there exists $\delta>0$
with $[t_0-\delta,t_0+\delta]\subset\mathcal{I}$ and a unique
solution $Y=Y^0$ of the problem \eqref{eq:IVP} with the values in
$\Gamma_s$ and such that $Y^0\in
C^1([t_0-\delta,t_0+\delta],\mat{\mathcal{X}}{s})$.

Let $\Delta$ consist of functions from
$C^1([t_0-\delta,t_0+\delta],\mat{\mathcal{X}}{s})_{\rm nc}$ with
values in $\Gamma$. Clearly, $Y^0\in\Delta$. Since $\Gamma$ is a
nc set, so is $\Delta$. We will show next that $\Delta$ is
uniformly open. Let $Y_*\in\Delta_{sr}$ for some $r\in\mathbb{N}$.
The function $Y_*$ is continuous, therefore the set
$Y_*([t_0-\delta,t_0+\delta])$ is a compact subset of
$\Gamma_{sr}$. For every $t\in [t_0-\delta,t_0+\delta]$, let
$\epsilon_t>0$ be such that $B_{\rm
nc}(Y_*(t),\epsilon_t)\subseteq\Gamma$. In particular, $B_{\rm
nc}(Y_*(t),\epsilon_t)_{sr}\subseteq\Gamma_{sr}$. The balls
$B_{\rm nc}(Y_*(t),\epsilon_t)_{sr}$,
$t\in[t_0-\delta,t_0+\delta]$, cover the compact set
$Y_*([t_0-\delta,t_0+\delta])$. Let
$t_1,\ldots,t_k\in[t_0-\delta,t_0+\delta]$ be such that the balls
$B_{\rm nc}(Y_*(t_i),\epsilon_{t_i})_{sr}$, $i=1,\ldots,k$, form a
finite sub-cover of $Y_*([t_0-\delta,t_0+\delta])$. Then, for
every $t\in[t_0-\delta,t_0+\delta]$, one has
$$\max_{i=1,\ldots,k}(\epsilon_{t_i}-\|Y_*(t)-Y_*(t_i)\|_{sr})>0.$$
Since the left-hand side of this inequality is continuous in $t$,
we have that
$$\mu:=\min_{t\in[t_0-\delta,t_0+\delta]}\max_{i=1,\ldots,k}(\epsilon_{t_i}-\|Y_*(t)-Y_*(t_i)\|_{sr})>0.$$
Let $Y\in B_{\rm nc}(Y_*,\mu)_{srm}$ for some $m\in\mathbb{N}$.
Then we have that, for every $t\in[t_0-\delta,t_0+\delta]$, there
is $i\in\{1,\ldots,k\}$ such that
\begin{multline*}
\|Y(t)-Y_*^{(m)}(t)\|_{srm}\le\|Y-Y_*^{(m)}\|_\infty\le
\|Y-Y_*^{(m)}\|_{srm}<\mu \\
\le \epsilon_{t_i}-\|Y_*(t)-Y_*(t_i)\|_{sr}.
\end{multline*}
Then
\begin{multline*}
\|Y(t)-Y_*^{(m)}(t_i)\|_{srm}\le
\|Y(t)-Y_*^{(m)}(t)\|_{srm}+\|Y_*^{(m)}(t)-Y_*^{(m)}(t_i)\|_{srm}\\
=
\|Y(t)-Y_*^{(m)}(t)\|_{srm}+\|Y_*(t)-Y_*(t_i)\|_{sr}<\epsilon_{t_i},
\end{multline*}
i.e., $Y(t)\in B_{\rm
nc}(Y_*(t_i),\epsilon_{t_i})\subseteq\Gamma$. Thus $Y_*$ has a
neighborhood $B_{\rm nc}(Y_*,\mu)$ contained in $\Delta$, so that
$\Delta$ is a uniformly open nc set.

Let
$$\Omega=\coprod_{m=1}^\infty(\Gamma_{sm}\times\Delta_{sm}).$$
Clearly, $\Omega$ is a uniformly open nc set in
$(\mat{\mathcal{X}}{s}\times
C^1([t_0-\delta,t_0+\delta],\mat{\mathcal{X}}{s}))_{\rm nc}$, and
$$F\colon\Omega\to
C^1([t_0-\delta,t_0+\delta],\mat{\mathcal{X}}{s})_{\rm nc}$$
defined by
$$[F(X,Y)](t):=Y(t)-X-\int_{t_0}^tg(\tau,Y(\tau))\,d\tau$$
is a continuous nc function. We note that \eqref{eq:IVP} is
equivalent to $F(X,Y)=0$. Since $\delta^Yg$ is continuous on
$[t_0-\delta,t_0+\delta]\times\Delta$ in the cb-norm, we can write
$$[\delta^YF(X,Y)(Z)](t)=Z(t)-\int_{t_0}^t\delta^Yg(\tau,Y(\tau))(Z(\tau))\,d\tau.$$
Clearly, $\delta^YF$ is cb-continuous on $\Omega$ and, in fact,
independent of $X$.

Our goal now is to show that
$\delta^YF(X^0,Y^0)\in\mathcal{L}_{\rm cb}(\mat{\mathcal{X}}{s})$
is invertible and its inverse
 is completely bounded, i.e., $(\delta^YF(X^0,Y^0))^{-1}\in\mathcal{L}_{\rm
cb}(\mat{\mathcal{X}}{s})$. Given any $m\in\mathbb{N}$ and $G\in
C^1([t_0-\delta,t_0+\delta],\mat{\mathcal{X}}{sm})$, there exists
a unique solution $Z\in
C^1([t_0-\delta,t_0+\delta],\mat{\mathcal{X}}{sm})$ of the
equation
\begin{equation}\label{eq:intsol}
\delta^YF(X^0,Y^0)^{(m)}(Z)=G,
\end{equation}
since \eqref{eq:intsol} is equivalent to the initial value problem
for the linear ODE
\begin{equation}\label{eq:linear-ODE}
\dot{Z}(t)-\delta^Yg(t,Y^0(t))^{(m)}(Z(t))=\dot{G}(t),\qquad
Z(t_0)=G(t_0),
\end{equation}
and the latter has a unique solution on $[t_0-\delta,t_0+\delta]$;
see, e.g., \cite[Theorem 10.6.3]{D}. Since the bounded operator
$\delta^YF(X^0,Y^0)^{(m)}$ is invertible, its inverse
$$(\delta^YF(X^0,Y^0)^{(m)})^{-1}=\Big((\delta^YF(X^0,Y^0))^{-1}\Big)^{(m)}$$
is bounded as well by the Banach open mapping theorem. We are
going now to estimate its norm. If \eqref{eq:intsol} holds, then
\begin{multline*}
\|Z\|_\infty\le\|G\|_\infty+\max_{t\in[t_0-\delta,t_0+\delta]}
\Big\|\int_{t_0}^t\delta^Yg(\tau,Y(\tau))^{(m)}(Z(\tau))\,d\tau\Big\|_{sm}\\
\le\|G\|_\infty+\delta\max_{t\in[t_0-\delta,t_0+\delta]}
\|\delta^Yg(t,Y(t))\|_{\mathcal{L}_{\rm
cb}(\mat{\mathcal{X}}{s})}\|Z\|_\infty.
\end{multline*}
Making $\delta$ smaller if necessary, so that
$$\kappa:=\max_{t\in[t_0-\delta,t_0+\delta]}\|\delta^Yg(t,Y(t))\|_{\mathcal{L}_{\rm
cb}(\mat{\mathcal{X}}{s})}<\frac{1}{\delta},$$ we obtain that
$$\|Z\|_\infty\le\frac{\|G\|_\infty}{1-\kappa\delta}\le\frac{\|G\|_{sm}}{1-\kappa\delta}.$$
Using this estimate, we obtain from \eqref{eq:linear-ODE} that
$$\|\dot{Z}\|_\infty\le\|\dot{G}\|_\infty+\kappa\|Z\|_\infty\le\Big(1+\frac{\kappa}{1-\kappa\delta}\Big)\|G\|_{sm}
=\frac{1-\kappa\delta+\kappa}{1-\kappa\delta}\|G\|_{sm}.$$
Therefore
$$\|Z\|_{sm}=\max\{\|Z\|_\infty,\|\dot{Z}\|_\infty\}\le
\max\Big\{\frac{1}{1-\kappa\delta},\frac{1-\kappa\delta+\kappa}{1-\kappa\delta}\Big\}\|G\|_{sm}.$$
If, moreover, $\delta\le 1$, then we obtain the inequality
$$\|Z\|_{sm}\le\frac{1-\kappa\delta+\kappa}{1-\kappa\delta}\|G\|_{sm},$$
i.e.,
$$\|(\delta^YF(X^0,Y^0)^{(m)})^{-1}
\|_{\mathcal{L}(\mat{\mathcal{X}}{sm})}\le\frac{1-\kappa\delta+\kappa}{1-\kappa\delta}.$$
Since the right-hand side is independent of $m$, the operator
$(\delta^YF(X^0,Y^0))^{-1}$ is completely bounded and
$$\|(\delta^YF(X^0,Y^0))^{-1}
\|_{\mathcal{L}_{\rm
cb}(\mat{\mathcal{X}}{s})}\le\frac{1-\kappa\delta+\kappa}{1-\kappa\delta}.$$

Now, as we fixed $\delta$ with the additional properties above,
all the assumptions of Theorem \ref{thm:ops-1} on the nc function
$F$ are satisfied, and then its conclusions I and II guarantee the
existence of a uniformly analytic nc function $f$ as in
\eqref{eq:uaf} which assigns to every initial data $X\in B_{\rm
nc}(X^0,\alpha)$ the unique solution $Y$ of \eqref{eq:IVP}.
\end{proof}

\subsection{Extremal problems with nc constraints}

In this section, we discuss extremal problems in nc spaces with
the constraints determined by nc functions. We use the implicit nc
function theorem to reduce the problem to certain equations which
give the necessary conditions for the constrained extremum. In the
case where the underlying nc spaces are over finite-dimensional
vector spaces, we obtain the equations with Lagrange multipliers.
Since these extremal problems are over matrices of infinitely many
sizes, the results are somewhat different from the classical
(commutative) case.

Let $\mathcal{X}$ and $\mathcal{Y}$ be operator spaces over
$\mathbb{R}$ \cite{Ruan}. Let
$\Omega\subseteq(\mathcal{X}\times\mathcal{Y})_{\rm nc}$ be a
uniformly open nc set, and let a function
$g\colon\Omega\to\mathbb{R}$ satisfy 
\begin{equation}\label{eq:dirsum-g}
g(X^{(m)},Y^{(m)})=g(X,Y)\ {\rm for\ every\ } X,Y\in\Omega \ {\rm
and}\ m\in\mathbb{N}.
\end{equation}
Let $\mathcal{W}$ be a real operator space and
$G\colon\Omega\to\mathcal{W}_{\rm nc}$ a nc function. Let
$\tau\colon\mathcal{W}_{\rm nc}\to\mathbb{R}$ satisfy
\begin{equation}\label{eq:dirsum-tau}
\tau(W^{(m)})=\tau(W)\ {\rm for\ every\ } W\in\mathcal{W}_{\rm nc}
\ {\rm and}\ m\in\mathbb{N}.
\end{equation}
Then $g=\tau\circ G\colon\Omega\to\mathbb{R}$ satisfies
\eqref{eq:dirsum-g}. One example of such $\tau$ is given by
$$\tau(W)=\|W\|_n,\qquad W\in\mathcal{W}^{n\times n},\
n\in\mathbb{N}.$$ Another example is a \emph{normalized trace on}
$\mathcal{W}_{\rm nc}$, defined as an arbitrary bounded linear
functional on $\mathcal{W}$ and then extended to square matrices
over $\mathcal{W}$ by
\begin{equation}\label{eq:trace}
\tau(W):=\frac{1}{n}\trace(\tau^{(n)}(W))=\frac{1}{n}\sum_{i=1}^n\tau(W_{ii}),\quad
W=[W_{ii}]\in\mat{\mathcal W}{n},\ n\in\mathbb{N}.
\end{equation}
Indeed, given $W=[W_{ii}]\in\mat{\mathcal W}{s}$, we have
\begin{multline*}
\tau(W^{(m)})=\frac{1}{sm}\trace\Big(\tau^{(sm)}(W^{(m)})\Big)=\frac{1}{sm}\sum_{i=1}^{sm}\tau([W^{(m)}]_{ii})\\
= \frac{1}{sm}\cdot m\sum_{i=1}^s\tau(W_{ii})=\tau(W).
\end{multline*}
Let $\mathcal{Z}$ be a real operator space and
$F\colon\Omega\to\mathcal{Z}_{\rm nc}$ a nc function. We will say
that $(X^0,Y^0)\in\Omega$ is a \emph{point of constrained (local)
maximum for $g$} if $g(X^0,Y^0)\ge g(X,Y)$ for every $(X,Y)$ from
some nc ball $B_{\rm nc}((X^0,Y^0),\epsilon)\subseteq\Omega$
subject to the constraint $F(X,Y)=0$.

Suppose $s\in\mathbb{N}$, $(X^0,Y^0)\in\Omega_s$ is a point of
constrained maximum  for $g$, and $g|_{\Omega_{sm}}$ is
G-differentiable at $(X^{0(m)},Y^{0(m)})$ for every
$m\in\mathbb{N}$. We also assume that $F$ satisfies the
assumptions of Theorem \ref{thm:ops-1} for the point $(X^0,Y^0)$.
Then by Theorem \ref{thm:ops-1} there exist $\alpha,\beta>0$ such
that $B(X^{0(m)},\alpha)\times
B(Y^{0(m)},\beta)\subseteq\Omega_{sm}$ for every $m\in\mathbb{N}$
and a uniformly analytic nc function $f\colon B_{\rm
nc}(X^0,\alpha)\to B_{\rm nc}(Y^0,\beta)$ so that $F(X,Y)=0$ if
and only if $Y=f(X)$. Then $X^0$ is a point of unconstrained local
maximum of the function $\Phi\colon B_{\rm nc}(X^0,\alpha)\to
\mathbb{R}$ defined by $$\Phi(X)=g(X,f(X)).$$ Restricting $\Phi$
to $B_{\rm nc}(X^0,\alpha)_{sm}=B(X^{0(m)},\alpha)$,
$m=1,2,\ldots$, we obtain a sequence of necessary conditions for
the extremum of $\Phi$:
\begin{equation*}
\delta\Phi(X^{0(m)})=0,\qquad m\in\mathbb{N}.
\end{equation*}
Taking into account conclusion III of Theorem \ref{thm:ops-1}, we
can write
\begin{multline*}
0=\delta\Phi(X^{0(m)})=\delta^Xg(X^{0(m)},Y^{0(m)})+\delta^Yg(X^{0(m)},Y^{0(m)})\delta
f(X^{0(m)})\\
=\delta^Xg(X^{0(m)},Y^{0(m)}) -\delta^Yg(X^{0(m)},Y^{0(m)})\\
\hfill \cdot (\delta^YF(X^{0(m)},Y^{0(m)}))^{-1}
\delta^XF(X^{0(m)},Y^{0(m)})\\
=\delta^Xg(X^{0(m)},Y^{0(m)})
-\delta^Yg(X^{0(m)},Y^{0(m)})\Big((\delta^YF(X^{0},Y^{0}))^{-1}\delta^XF(X^{0},Y^{0})\Big)^{(m)}.
\end{multline*}
In the special case of $g=\tau\circ G$, where $G\colon\Omega\to
\mathcal{W}_{\rm nc}$ is a nc function which is G-differentiable
at $(X^0,Y^0)$ and $\tau$ is a normalized trace on
$\mathcal{W}_{\rm nc}$, we have, for every
$H=[H_{ii}]\in\mat{(\mat{\mathcal{X}}{s})}{m}\cong
\mat{\mathcal{X}}{sm}$, that
\begin{multline*}
\delta^Xg(X^{0(m)},Y^{0(m)})(H)=\tau(\delta^XG(X^{0(m)},Y^{0(m)})(H))=\tau(\delta^XG(X^{0},Y^{0})^{(m)}(H))\\
=\frac{1}{m}\sum_{i=1}^m\tau(\delta^XG(X^{0},Y^{0})(H_{ii}))=\frac{1}{m}\sum_{i=1}^m\delta^Xg(X^{0},Y^{0})(H_{ii}),
\end{multline*}
and similarly for
$H=[H_{ii}]\in\mat{(\mat{\mathcal{Y}}{s})}{m}\cong
\mat{\mathcal{Y}}{sm}$,
\begin{multline*}
\delta^Yg(X^{0(m)},Y^{0(m)})(H)=\tau(\delta^YG(X^{0(m)},Y^{0(m)})(H))=\tau(\delta^YG(X^{0},Y^{0})^{(m)}(H))\\
=\frac{1}{m}\sum_{i=1}^m\tau(\delta^YG(X^{0},Y^{0})(H_{ii}))=\frac{1}{m}\sum_{i=1}^m\delta^Yg(X^{0},Y^{0})(H_{ii}).
\end{multline*}
Then it follows that, in this special case, the sequence of the
necessary conditions for the extremum above, with $m=1,2,\ldots,$
is equivalent to the single necessary condition for $m=1$.

We now summarize the discussion above in the following theorem.
\begin{thma}\label{thm:constr-max}
Let $s\in\mathbb{N}$, let $(X^0,Y^0)\in\Omega_s$ be a point of
constrained maximum for a function $g\colon\Omega\to\mathbb{R}$
satisfying \eqref{eq:dirsum-g}. If the nc function
$F\colon\Omega\to\mathcal{Z}_{\rm nc}$ that determines the
constraint satisfies the assumptions of Theorem \ref{thm:ops-1}
for the point $(X^0,Y^0)$ and $g|_{\Omega_{sm}}$ is
G-differentiable at $(X^{0(m)},Y^{0(m)})$ for every
$m\in\mathbb{N}$, then
\begin{multline}\label{eq:seq}
\delta^Xg(X^{0(m)},Y^{0(m)})
=\delta^Yg(X^{0(m)},Y^{0(m)})\Big((\delta^YF(X^{0},Y^{0}))^{-1}\delta^XF(X^{0},Y^{0})\Big)^{(m)},\\
m\in\mathbb{N}.
\end{multline}
If, in addition, $g=\tau\circ G$, where $G\colon\Omega\to
\mathcal{W}_{\rm nc}$ is a nc function which is G-differentiable
at $(X^0,Y^0)$ and $\tau$ is a normalized trace on
$\mathcal{W}_{\rm nc}$, then the sequence of equations
\eqref{eq:seq} is equivalent to the single equation
\begin{equation}\label{eq:single}
\delta^Xg(X^{0},Y^{0})
=\delta^Yg(X^{0},Y^{0})(\delta^YF(X^{0},Y^{0}))^{-1}\delta^XF(X^{0},Y^{0}).
\end{equation}
\end{thma}

We note that, in the special case of $g=\tau\circ G$ that we
singled out in Theorem \ref{thm:constr-max}, one can restrict the
search for a point of constrained maximum to the \emph{critical
points}, i.e., those points $(X^0,Y^0)$ which solve
\eqref{eq:single} together with $F(X^0,Y^0)=0$.

We now stick to the special case of $g=\tau\circ G$ above and
assume that $\mathcal{X}=\mathbb{R}^a$,
$\mathcal{Y}=\mathcal{Z}=\mathbb{R}^b$, and $\mathcal{W}$ is an
arbitrary real operator space. We will show that the critical
points in this case can be found by the method of Lagrange
multipliers.

We first consider the case where a point of constrained maximum is
scalar, i.e., $s=1$ and
$$(x^0,y^0)=(x_1^0,\ldots,x_a^0,y_1^0,\ldots,y_b^0)\in\Omega_1\subseteq\mathbb{R}^{a+b}.$$
Then $\delta^YF(x^0,y^0)$ can be identified with the invertible
matrix $[\frac{\partial F_i}{\partial
y_j}(x^0,y^0)]\in\mat{\mathbb{R}}{b}$, and \eqref{eq:single} is
equivalent to the equations
\begin{equation}\label{eq:eqs}
\frac{\partial g}{\partial
x_i}(x^0,y^0)=\sum_{j=1}^b\sum_{k=1}^b\frac{\partial g}{\partial
y_j}(x^0,y^0)\Big((\delta^YF(x^0,y^0))^{-1}\Big)_{jk}\frac{\partial
F_k}{\partial x_i}(x^0,y^0),\quad i=1,\ldots,a.
\end{equation}
Setting $$\lambda_k=-\sum_{j=1}^b\frac{\partial g}{\partial
y_j}(x^0,y^0)\Big((\delta^YF(x^0,y^0))^{-1}\Big)_{jk},\quad
k=1,\ldots,b,$$ we can rewrite \eqref{eq:eqs} as
$$\frac{\partial g}{\partial
x_i}(x^0,y^0)+\sum_{k=1}^b\lambda_k\frac{\partial F_k}{\partial
x_i}(x^0,y^0)=0,\quad i=1,\ldots,a.
$$
On the other hand, for $i=1,\ldots,b$, we have
\begin{multline*}
\frac{\partial g}{\partial
y_i}(x^0,y^0)+\sum_{k=1}^b\lambda_k\frac{\partial F_k}{\partial
y_i}(x^0,y^0)\\
=\frac{\partial g}{\partial
y_i}(x^0,y^0)-\sum_{j=1}^b\sum_{k=1}^b\frac{\partial g}{\partial
y_j}(x^0,y^0)\Big((\delta^YF(x^0,y^0))^{-1}\Big)_{jk}\frac{\partial
F_k}{\partial y_i}(x^0,y^0)\\
=\frac{\partial g}{\partial
y_i}(x^0,y^0)-\sum_{j=1}^b\frac{\partial g}{\partial
y_j}(x^0,y^0)\delta_{ij}=0.
\end{multline*}
Thus we obtain the full set of $a+2b$ equations for a constrained
maximum, with $a+2b$ unknowns $x^0_1$, \ldots, $x^0_a$, $y^0_1$,
\ldots, $y^0_b$, $\lambda_1$, \ldots, $\lambda_k$:
\begin{eqnarray*}
\frac{\partial g}{\partial
x_i}(x^0,y^0)+\sum_{k=1}^b\lambda_k\frac{\partial F_k}{\partial
x_i}(x^0,y^0)=0, & i=1,\ldots,a,\\
\frac{\partial g}{\partial
y_j}(x^0,y^0)+\sum_{k=1}^b\lambda_k\frac{\partial F_k}{\partial
y_j}(x^0,y^0)=0, & j=1,\ldots,b,\\
F_j(x^0,y^0)=0, & j=1,\ldots,b.
\end{eqnarray*}

In the case of a constrained maximum at $(X^0,Y^0)\in\Omega_s$
with $s>1$, a similar calculation gives the equations with
matrices $\Lambda_k\in\mat{\mathbb{R}}{s}$, $k=1,\ldots, b$, and
$s\times s$ matrix values of $F_k(X^0,Y^0)$ (i.e., $(a+2b)s^2$
equations with $(a+2b)s^2$ unknowns):
\begin{eqnarray*}
\frac{\partial g}{\partial
(X_i)_{\alpha\beta}}(X^0,Y^0)+\sum_{k=1}^b\trace\Big(\Lambda_k\frac{\partial
F_k}{\partial
(X_i)_{\alpha\beta}}(X^0,Y^0)\Big)=0, \\
i=1,\ldots,a,\ \alpha,\beta=1,\ldots,s,\\
\frac{\partial g}{\partial
(Y_j)_{\alpha\beta}}(X^0,Y^0)+\sum_{k=1}^b\trace\Big(\Lambda_k\frac{\partial
F_k}{\partial (Y_j)_{\alpha\beta}}(X^0,Y^0)\Big)=0,
 \\
 j=1,\ldots,b,\ \alpha,\beta=1,\ldots,s,\\
(F_j(X^0,Y^0))_{\alpha\beta}=0, \quad j=1,\ldots,b,\
\alpha,\beta=1,\ldots,s.
\end{eqnarray*}

In general, the search for a point of constrained maximum for $g$
requires solving these equations for all matrix sizes
$s\in\mathbb{N}$. Thus, for a problem at hand, an additional
insight would be valuable in order to restrict the search to
certain matrix sizes.


\begin{thebibliography}{9}



\bibitem{AK} G. Abduvalieva and D.~S. Kaliuzhnyi-Verbovetskyi, Fixed
point theorems for noncommutative functions, \emph{J. Math. Anal.
Appl.} 401 (2013), no. 1, 436--446.

\bibitem{Ag-Mc1} J. Agler and J. E. McCarthy, Global holomorphic functions in several non-commuting variables,
\emph{Canadian J. Math.} 67 (2015), no. 2, 241--285.

\bibitem{Ag-Mc2} J. Agler and J. E. McCarthy, The implicit function theorem and free algebraic sets,
\emph{Trans. Amer. Math. Soc.}, to appear.

\bibitem{Ag-Y} J. Agler and N. J. Young, Symmetric functions of two noncommuting variables,
 J. Functional Analysis 266 (2014), pp. 5709--5732.

\bibitem{D} J. Dieudonn\'{e},  Foundations of modern analysis, Pure and Applied Mathematics, Vol. X, Academic Press,
  New York--London 1960, xiv+361 pp.

\bibitem{ER} E. G. Effros and Zh.-J. Ruan, Operator spaces, London Mathematical Society Monographs. New Series, 23.
 The Clarendon Press, Oxford University Press, New York, 2000, xvi+363 pp.

\bibitem{Helton} J. W. Helton, Manipulating matrix inequalities automatically. Mathematical systems theory in biology,
 communications,
computation, and finance (Notre Dame, IN, 2002), 237--256, IMA Vol. Math. Appl., 134, Springer, New York, 2003.


\bibitem{HKMcC1} J. W. Helton, I. Klep, and S. A. McCullough, Proper analytic free maps. \emph{J. Funct. Anal.}
260 (2011),
 no. 5, 1476--1490.

\bibitem{HKMcC2} J. W. Helton, I. Klep, and S. A. McCullough,  Analytic mappings between noncommutative pencil
balls,
\emph{J. Math. Anal. Appl.} 376 (2011), no. 2, 407--428.

\bibitem{HMcC} J. W. Helton and S. A. McCullough, Every convex free basic semi-algebraic set has an LMI
representation,
 \emph{Ann. of Math.} 176 (2012), no. 2, 979-1013


\bibitem{HP} J. W. Helton and M. Putinar, Positive polynomials in scalar and matrix variables,
the spectral theorem, and optimization, Operator theory,
structured matrices, and dilations, 229--306, Theta Ser. Adv.
Math., 7, Theta, Bucharest, 2007.



\bibitem{KVV2}  D. S. Kaliuzhnyi-Verbovetskyi and V. Vinnikov, Noncommutative rational functions, their
difference-differential
 calculus and realizations, \emph{Multidimens. Syst. Signal Process.} 23 (2012), no. 1--2, 49--77.

\bibitem{KV-VV} D. S.  Kaliuzhnyi-Verbovetskyi and V. Vinnikov, Foundations of Free  Non-commutative
Function Theory, Math Surveys and Monographs, Vol. 199, AMS, 2014,
183 pp.

\bibitem{KVV1} D. S. Kaliuzhnyi-Verbovetskyi and V. Vinnikov, Singularities of rational functions and minimal
factorizations: the noncommutative and the commutative setting,
\emph{Linear Algebra Appl.} 430 (2009), no. 4, 869--889.

\bibitem{Kant-Ak} L.B. Kantorovich and G.P. Akilov, Functional Analysis, 2nd ed., Nauka, Moscow, 1977.

\bibitem{MS} P. Muhly and B. Solel, Progress in noncommutative function
theory,
\emph{Science China Mathematics} 54 (2011), no. 11, 2275--2294.

\bibitem{Pascoe2} J. E. Pascoe, The inverse function theorem and  the Jacobian conjecture for free analysis,
 \emph{Math. Z.} 278 (2014), 987--994.


\bibitem{Paulsen} V. Paulsen. Completely bounded maps and operator algebras,
 Cambridge Studies in Advanced Mathematics,78, Cambridge University Press, Cambridge, 2002.

\bibitem{G.Pi} G. Pisier. Introduction to operator space theory,
 London Mathematical Society Lecture Note Series, 294,  Cambridge University Press, Cambridge, 2003.

\bibitem{Pop1} G. Popescu, Free holomorphic functions on the unit ball of $B(\mathcal{H})^n$, \emph{J. Funct. Anal.}
 241 (2006), no. 1, 268--333.

\bibitem{Pop2} G. Popescu,  Free holomorphic functions on the unit ball of $B(\mathcal{H})^n$, II, \emph{J. Funct. Anal.}
 258 (2010), no. 5, 1513--1578.


\bibitem{Ruan}  Zh. J. Ruan, On real operator spaces, International Workshop on
Operator Algebra and Operator Theory (Linfen, 2001), \emph{Acta
Math. Sin. (Engl. Ser.)} 19 (2003), no. 3, 485--496.

\bibitem{JLTaylor} J. L. Taylor, A general framework for a multi-operator functional calculus, \emph{ Advances in Math.}
 9 (1972),  183--252.

\bibitem{JLTaylor-nc}  J. L. Taylor, Functions of several noncommuting
variables,
\emph{ Bull. Amer. Math. Soc.,} 79 (1973),  1--34.

\bibitem{Voic1} D. Voiculescu, Free analysis questions,
 I. Duality transform for the coalgebra of $\partial_{X:B}$, \emph{Int. Math. Res. Not.} 16 (2004), 793--822.

\bibitem{Voic2} D.-V. Voiculescu,   Free analysis questions II:
the Grassmannian completion and the series expansions at the
origin, \emph{J. Reine Angew. Math.} 645 (2010), 155--236.

\end{thebibliography}
\end{document}